\documentclass[12pt,reqno]{amsart}

\usepackage{amsmath,amssymb,mathrsfs,amsfonts,amsthm,mathptmx}
\usepackage{graphicx,cite,times,color,bm,dcolumn}
\usepackage{mathrsfs}
\usepackage{cases}
\usepackage{psfrag}
\usepackage{booktabs}

\usepackage[top=1.2in, bottom=1.2in, left=1in, right=1in]{geometry}

\numberwithin{equation}{section}
\numberwithin{figure}{section}
\numberwithin{table}{section}
\newtheorem{theorem}{Theorem}[section]
\newtheorem{lemma}{Lemma}[section]
\newtheorem{definition}{Definition}[section]
\newtheorem{corollary}{Corollary}[section]
\newtheorem{proposition}{Proposition}[section]
\newtheorem{remark}{Remark}[section]
\newtheorem{assumption}{Assumption}[section]

\newtheorem{example}{Example}[section]
\allowdisplaybreaks \allowdisplaybreaks[4]

\begin{document}
	\title[Runge-Kutta semidiscretizations for stochastic Maxwell equations]{Runge-Kutta semidiscretizations for stochastic  Maxwell equations with additive noise}
		
	\author{Chuchu Chen}
	\address{LSEC, ICMSEC,  Academy of Mathematics and Systems Science, Chinese Academy of Sciences, Beijing 100190, China; \and	
		School of Mathematical Sciences, University of Chinese Academy of Sciences, Beijing 100049, China}
	\email{chenchuchu@lsec.cc.ac.cn}
	
		\author{Jialin Hong}
	\address{LSEC, ICMSEC,  Academy of Mathematics and Systems Science, Chinese Academy of Sciences, Beijing 100190, China;
\and	
School of Mathematical Sciences, University of Chinese Academy of Sciences, Beijing 100049, China}
	\email{hjl@lsec.cc.ac.cn}

\author{Lihai Ji}
\address{Institute of Applied Physics and Computational Mathematics, Beijing 100094, China}
\email{jilihai@lsec.cc.ac.cn (Corresponding author)}

	\thanks{
		The research of C. Chen and J. Hong were supported by the NNSFC (NOs. 91130003, 11021101, 11290142, and 91630312), the research of
		L. Ji was supported by the NNSFC (NOs. 11601032, and 11471310).}

\maketitle	
	\begin{abstract}
{
%
 The paper concerns semidiscretizations in time of stochastic Maxwell equations driven by additive noise. We show that the equations admit  physical properties and mathematical structures, including regularity, energy and divergence evolution laws,   and stochastic symplecticity, etc. In order to inherit the intrinsic properties of the original system, we introduce a general class of stochastic Runge-Kutta methods, and deduce the condition of  symplecticity-preserving.
 By utilizing a priori estimates on numerical approximations and semigroup approach, we show that the methods, which are algebraically stable and coercive, are well-posed and convergent with order one in mean-square sense,
 which answers an open problem in \cite{CH2016} for stochastic Maxwell equations driven by additive noise.
}
	\end{abstract}
	
\section{Introduction}

Consider the following semilinear stochastic Maxwell equations with additive noise,
\begin{equation}\label{sto_max}
\begin{cases}
\varepsilon {\rm d}{\bf E}-\nabla\times {\bf H}{\rm d}t=-{\bf J}_{e}(t,{\bf x},{\bf E},{\bf H}){\rm d}t-{\bf J}_e^{r}(t,{\bf x})\circ{\rm d}W(t),~ &(t,{\bf x})\in(0,~T]\times D,\\
\mu {\rm d}{\bf H}+\nabla\times {\bf E}{\rm d}t=-{\bf J}_{m}(t,{\bf x},{\bf E},{\bf H}){\rm d}t-{\bf J}_m^{r}(t,{\bf x})\circ{\rm d}W(t),~ &(t,{\bf x})\in(0,~T]\times D,\\
{\bf E}(0,{\bf x})={\bf E}_0({\bf x}),~{\bf H}(0,{\bf x})={\bf H}_0({\bf x}),~&{\bf x}\in D,\\
{\bf n}\times {\bf E}={\bf 0},~&(t,{\bf x})\in(0,~T]\times\partial D,
\end{cases}
\end{equation}
where ${\bf E}$ is the electric field, ${\bf H}$ is the magnetic field, ${\varepsilon}$ denotes the permittivity, $\mu$ denotes the permeability satisfying $\varepsilon,\mu\in L^{\infty}(D)$, $\varepsilon,\mu\geq  \delta>0$.
Here $\circ$ means  Stratonovich integral, $D\subset {\mathbb R}^{3}$ is a bounded domain, $T\in(0,~\infty)$, and the function  ${\bf J}:[0,T]\times D\times {\mathbb R}^3\times{\mathbb R}^3\to{\mathbb R}^3$ is a continuous function satisfying
\begin{eqnarray}
&&|{\bf J}(t,{\bf x},u,v)|\leq L(1+|u|+|v|),\label{bound J}\\
&&|{\bf J}(t,{\bf x},u_1,v_1)-{\bf J}(s,{\bf x},u_2,v_2)|\leq L(|t-s|+|u_1-u_2|+|v_1-v_2|),\label{bound partialJ}
\end{eqnarray}
for all ${\bf x}\in D$, $u,v,u_1,v_1,u_2,v_2\in{\mathbb R}^3$,  the constant $L>0$. Here  $|\cdot|$ denotes the Euclidean norm, and ${\bf J}$ could be ${\bf J}_e$ or ${\bf J}_m$, and the function ${\bf J}^r:[0,T]\times D\to {\mathbb R}^3$ is a continuous bounded function
with ${\bf J}^{r}$ being ${\bf J}^r_e$ or ${\bf J}_m^r$. Throughout this paper, $W(t)$ is a $Q$-Wiener process with respect to a filtered probability space $(\Omega,{\mathcal F},\{{\mathcal F}_{t}\}_{0\leq t\leq T},{\mathbb P})$ with $Q$ being a  symmetric, positive definite operator with finite trace on $U=L^2(D)$.
If we denote an orthonormal basis of the space $U$ by $\{e_i\}_{i\in{\mathbb N}}$, then $W(t)$
can be represented as
\begin{equation}
W(t)=\sum_{i=1}^{\infty}Q^{\frac12}e_i\beta_i(t),~t\in[0,~T],
\end{equation}
where $\{\beta_i(t)\}_{i\in{\mathbb N}}$ is a sequence of independent real-valued Brownian motions.

 The well-posedness of stochastic Maxwell equations has been investigated by semigroup approach in  \cite{LSY2010,CHJ2018}, by a refined  Faedo-Galerkin method and spectral multiplier theorem in \cite{Hor2017}, by using the  stochastically perturbed PDEs approach in \cite{SW2017}. The regularity of the
 solution of stochastic Maxwell equations driven by It\^o multiplicative noise is considered in \cite{CHJ2018}, allowing sufficient spatial smoothness on the coefficients and noise term. The stochastic multi-symplectic structures are investigated in \cite{HJZ2014,CHZ2016} for stochastic Maxwell equations driven by additive noise via different approaches, in \cite{HJZC2017} for stochastic Maxwell equations driven by multiplicative noise.

  The numerical analysis of stochastic Maxwell equations is a recent active ongoing
research subject. There are now a certain number of papers devoted to this field but many problems still need to be solved (see e.g. \cite{Zhang2008,BAZC2010,HJZ2014,CHZ2016,HJZC2017,CHJ2018} and references therein). Particularly, \cite{HJZ2014} proposes a stochastic multi-symplectic method for stochastic Maxwell equations with additive noise based on the stochastic version of variational principle, which has the merits of preserving the discrete stochastic multi-symplectic conservation law and stochastic energy dissipative properties. In \cite{CHZ2016}, the comparison of three different stochastic multi-symplectic methods and the analysis of the linear growth property of energy and the conservative property of divergence are studied. In \cite{HJZC2017}, the authors constructed an innovative stochastic multi-symplectic energy-conserving method for three dimension stochastic Maxwell equations with multiplicative noise by using wavelet interpolate technique. For the rigorous convergence analysis of numerical approximations, we refer to the very recently work \cite{CHJ2018}, in which  mean-square convergence  of a semi-implicit Euler scheme for stochastic Maxwell equations with multiplicative It\^o noise is investigated. Via the energy estimate technique and a priori estimates on exact and numerical solutions, authors show that the method is convergent with order $1/2$.

To the best of our knowledge, however, there has been no work in the literature which considers the  infinite-dimensional stochastic Hamiltonian system form, stochastic symplecticity for stochastic Maxwell equations.
By introducing two new Hamiltonian functionals,
 and by utilizing the properties of variational integrals, we present  stochastic Maxwell equations \eqref{sto_max} as the equivalent infinite-dimensional stochastic Hamiltonian system form directly. As a result, the phase flow of  equations \eqref{sto_max}
 preserves the symplectic structure  $\overline{\omega}(t)=\int_{D}{\rm d}{\bf E}(t,{\bf x})\wedge {\rm d}{\bf H}(t,{\bf x}){\rm d}x$ almost surely.
 Meanwhile, we present the regularity in the space ${\mathcal D}(M^k)$ ($k\in{\mathbb N}$) of the solution for stochastic Maxwell equations \eqref{sto_max}, where $M$ denotes the Maxwell operator.  This regularity, together with the adaptedness to filtration, yields the H\"older continuity of the solution in the space ${\mathcal D}(M^{k-1})$ both  in mean-square and in mean senses. Furthermore, the evolution laws of energy and divergence are also investigated via the formal application of It\^o formula.

It is important to design numerical methods which could preserve  the intrinsic properties of the original system as much as possible, due to the superiority on the long time simulation and stability etc.
In order to construct stochastic symplectic methods for stochastic Maxwell equations \eqref{sto_max}, we introduce a general class of stochastic Runge-Kutta methods  to these equations in temporal direction. By utilizing the structure of numerical methods and the properties of differential 2-forms, we derive the symplectic conditions of coefficients for the  methods to preserve stochastic symplectic structure. The existence and uniqueness  of the numerical solution  are proved for the general class of stochastic Runge-Kutta methods which is algebraically stable and coercive.
The relevant prerequisite for the mean-square convergence analysis is to provide the regularity in the space ${\mathcal D}(M^k)$ and H\"older continuity in the space ${\mathcal D}(M^{k-1})$ for the
original system, and also for the temporal stochastic Runge-Kutta semidiscretizations. To deal with the difficulty caused by the interaction of the unbounded operator $M$, stochastic terms and the complex structure of Runge-Kutta method,
we  make use of the
semigroup approach which makes the mild solution can be expressed in the form containing a bounded linear semigroup instead of the unbounded differential operator, and a priori estimate  on the operators and semigroup, as well as the coercivity and algebraic stability of the proposed methods.
These estimates are then essential for the error analysis, which allow to establish optimal mean-square convergence rates (see Theorem 4.3). An immediate consequence of this result is that the order of mean-square convergence is $1$, which answers an open problem in \cite{CH2016} for stochastic Maxwell equations driven by additive noise. The analysis holds for
 the algebraically stable and coercive stochastic Runge-Kutta methods.
 Note that symplectic Runge-Kutta methods are algebraic stable automatically, as a consequence the mean-square convergence order of the coercive symplectic Runge-Kutta methods is  $1$.

The paper is organized as follows: in Section 2, some preliminaries are collected and an abstract formulation of \eqref{sto_max} is set forth.  Some properties of stochastic Maxwell equations, including regularity, evolution laws of energy and divergence are also considered. Section 3 is devoted to the stochastic symplecticity of stochastic Maxwell equations. In Section 4, a semi-discrete scheme is proposed and our main results are stated: in Section 4.1 we give some conditions to guarantee that a given stochastic Runge-Kutta method is symplectic; in Section 4.2 we show the unique existence and regularity of numerical solution of general  stochastic Runge-Kutta method. Section 4.3 is devoted to the proof of the convergence theorem of stochastic Runge-Kutta methods satisfying the definition of algebraical stability and coercivity condition.

\section{Preliminaries and framework}
\subsection{Notations}
Throughout the paper, we will use the following notations.
\begin{itemize}
 	\item[1.] We will work with the real Hilbert space ${\mathbb H}=L^2(D)^3\times L^2(D)^3$, endowed with the inner product
 	\[
 	\left\langle \begin{pmatrix}
 	{\bf E}_1\\{\bf H}_1
 	\end{pmatrix},~ \begin{pmatrix}
 	{\bf E}_2\\{\bf H}_2
 	\end{pmatrix}\right\rangle_{\mathbb H}=\int_{D}(\varepsilon {\bf E}_1\cdot {\bf E}_2
 	+\mu{\bf H}_1\cdot{\bf H}_2){\rm d}{\bf x}
 	\]
for all ${\bf E}_1, {\bf H}_1,{\bf E}_2,{\bf H}_2\in L^2(D)^3$, and the norm
 	\[
 	\left\|\begin{pmatrix}
 	{\bf E}\\{\bf H}
 	\end{pmatrix}\right\|_{\mathbb H}=\left[\int_{D}\left(\varepsilon|{\bf E}|^2+\mu|{\bf H}|^2\right){\rm d}{\bf x}\right]^{1/2},\quad \forall~{\bf E}, {\bf H}\in  L^2(D)^3.
 	\]
 	\item[2.] We will denote  the Maxwell operator by
 	\begin{equation}\label{M_operator}
 	M=\begin{pmatrix}
 	0& \varepsilon^{-1}\nabla\times \\
 	-\mu^{-1}\nabla\times &0 \\
 	\end{pmatrix}
 	\end{equation}
 	with domain
 	\begin{equation}
 	\begin{split}
 	{\mathcal D}(M)&=\left\{\begin{pmatrix}
 	{\bf E} \\
 	{\bf H}
 	\end{pmatrix}\in {\mathbb H}:~M\begin{pmatrix}
 	{\bf E} \\
 	{\bf H}
 	\end{pmatrix}=\begin{pmatrix}
 	\varepsilon^{-1} \nabla\times{\bf H}\\
 	-\mu^{-1}\nabla\times{\bf E}
 	\end{pmatrix}\in{\mathbb H},~ {\bf n}\times{\bf E}\Big|_{\partial D}={\bf 0} \right\}\\[2mm]
 	&=H_0({\rm curl},D)\times H({\rm curl},D),
 	\end{split}
 	\end{equation}
 where the curl-spaces are defined by
 \begin{equation*}
   \begin{split}
     H({\rm curl},D):&=\{ v\in L^2(D)^3:~\nabla\times v\in L^2(D)^3 \},\\[2mm]
     H_0({\rm curl},D):&=\{ v\in H({\rm curl},D):~{\bf n}\times v|_{\partial D}={\bf 0} \}.
   \end{split}
 \end{equation*}
 	The corresponding graph norm is $\|v\|_{{\mathcal D}(M)}:=\left(\|v\|_{\mathbb H}^2+\|Mv\|_{\mathbb H}^2\right)^{1/2}$.
 	A frequently used property for Maxwell operator $M$ is:
 	$
 	\langle Mu,~u\rangle_{\mathbb H}=0, ~\forall~u\in{\mathcal D}(M).
 	$
 	\item[3.] The Maxwell operator $M$ defined in \eqref{M_operator} is closed, skew-adjoint on $\mathbb{H}$, and thus  generates a unitary $C_0$-group $S(t)=e^{tM}$ on $\mathbb{H}$ in the view of Stone's theorem. 
 	A frequently used tool of semigroup is the following estimate (see \cite[Lemma 3.1]{CHJ2018}):
 		\begin{equation}
 	\|S(t)-Id\|_{{\mathcal L}({\mathcal D}(M);{\mathbb H})}\leq Ct,
 	\end{equation}
 	where the constant $C$ does not depend on $t$.
 	
 	\item[4.]
 	 We define the space ${\mathcal D}(M^n)$ by the domain of the $n$-th power of operator $M$ for $n\in{\mathbb N}$, with norm
 \[
 	\|u\|_{{\mathcal D}(M^n)}:=\left(\|u\|_{\mathbb H}^2+\|M^n u\|_{\mathbb H}^2\right)^{1/2}.
 	\]
  In fact,
 	the norm $\|\cdot\|_{{\mathcal D}(M^n)}$ corresponds to the scalar product
 	\[
 	\langle u,~v\rangle_{{\mathcal D}(M^n)}=\langle u,~v\rangle_{\mathbb H}
 	+\langle M^nu,~M^nv\rangle_{\mathbb H}.
 	\]
 	Moreover, we know that $\|u\|_{{\mathcal D}(M^n)}\leq C\|u\|_{{\mathcal D}(M^m)}$ for all $u\in{\mathcal D}(M^m)$, $n\leq m$.
 	
 	\item[5.] Denote $HS(U,H)$  the Banach space of all Hilbert-Schmidt operators from one separable Hilbert space $U$ to another separable Hilbert space $H$, equipped with the norm
 	\[
 	\|\Gamma\|_{HS(U,H)}=\left(\sum_{j=1}^{\infty}\|\Gamma\eta_j\|_{H}^2\right)^{\frac12},
 	\]
 	where $\{\eta_j\}_{j\in{\mathbb N}}$ is any orthonormal basis of $U$.

 \item[6.] Throughout this paper, $C$ will denote various constants. The same symbol will be used for different constants. When it is necessary to indicate that a constant depends on some parameters, we will use the notation $C(\cdot)$. For instance, $C(T, p)$ is a constant depending on $T$ and $p$.
 \end{itemize}

\subsection{Framework}
We work on the abstract form of stochastic Maxwell equations in infinite dimensional space ${\mathbb H}$:
\begin{equation}\label{sM_equations}
\begin{cases}
{\rm d}u(t)=\left[Mu(t)+F(t,u(t))\right]{\rm d}t+B(t){\rm d}W(t),~t\in(0,~T],\\
u(0)=u_0,
\end{cases}
\end{equation}
where $u(t)=({\bf E}^T(t),{\bf H}^T(t))^T$, $u_0=({\bf E}_0^T,{\bf H}_0^T)$. Here $F:[0,~T]\times {\mathbb H}\to{\mathbb H}$ is a Nemytskij operator associated to ${\bf J}_{e}$, ${\bf J}_m$, which is defined by
\begin{equation}\label{F}
F(t,u(t))({\bf x})=\left( \begin{array}{c}
-\varepsilon^{-1}{\bf J}_{e}(t,{\bf x},{\bf E}(t,{\bf x}),{\bf H}(t,{\bf x}))\\
-\mu^{-1}{\bf J}_{m}(t,{\bf x},{\bf E}(t,{\bf x}),{\bf H}(t,{\bf x}))
\end{array} \right), ~t\in[0,T],~{\bf x}\in D,~u(t)\in{\mathbb H}.
\end{equation}
For diffusion term, we introduce the Nemytskij operator $B:[0,~T]\to HS(U_0,{\mathbb H})$ by
\begin{equation}
(B(t)v)({\bf x})=\left( \begin{array}{c}
-\varepsilon^{-1}{\bf J}_{e}^{r}(t,{\bf x})v({\bf x})\\
-\mu^{-1}{\bf J}_{m}^r(t,{\bf x})v({\bf x})
\end{array} \right),\quad {\bf x}\in D \text{ and } v\in U_0:=Q^{\frac12}U.
\end{equation}

\subsubsection{Well-posedness and regularity}
First we present the well-posedness in the Hilbert space ${\mathbb H}$ of the stochastic Maxwell equations \eqref{sM_equations}. From \cite{CHJ2018}, we know that conditions \eqref{bound J} and \eqref{bound partialJ} yield
the linear growth and global Lipschitz properties of the function $F$, i.e., there exists a constant $C$ depending on $\delta$, the volume $|D|$ of the domain $D$ and the constant $L$ in \eqref{bound J} and \eqref{bound partialJ}, such that
\begin{eqnarray}
&&\|F(t,u)\|_{\mathbb H}\leq C\big(1+\|u\|_{\mathbb H}\big),\\
&&\|F(t,u)-F(s,v)\|_{\mathbb H}\leq C\big(|t-s|+\|u-v\|_{\mathbb H}\big),
\end{eqnarray}
for all $t,s\in [0,T]$ and $u,v\in{\mathbb H}$.

The following proposition gives the existence and uniqueness of
the mild solution of equation \eqref{sM_equations}, which has been discussed for example in \cite{LSY2010,SW2017,CHJ2018}.
\begin{proposition}\label{wellposedness thm}
	Suppose conditions \eqref{bound J} and \eqref{bound partialJ} are fulfilled, and let $W(t)$, $t\in[0,~T]$ be a $Q$-Wiener process with $Q$ being symmetric, positive definite and with finite trace, and let $u_0$ be an ${\mathcal F}_0$-measurable ${\mathbb H}$-valued random variable satisfying $\|u_0\|_{L^{p}(\Omega;{\mathbb H})}<\infty$ for some $p\geq 2$. Then stochastic Maxwell equations \eqref{sM_equations} have a unique mild solution given by
	\begin{equation}\label{mild sol}
	u(t)=S(t)u_0+\int_0^t S(t-s)F(s,u(s)){\rm d}s+\int_0^t S(t-s)B(s)dW(s)\quad {\mathbb P}\text{-}a.s.
	\end{equation}
	for each $t\in[0,~T]$.
	
	Moreover, there exists a constant $C:=C(p,T,{\rm tr}(Q))\in(0,~\infty)$ such that
	\begin{equation}
	\sup_{t\in[0,~T]}{\mathbb E}\|u(t)\|^{p}_{ {\mathbb H}}\leq C(1+\|u_0\|^p_{L^{p}(\Omega;{\mathbb H})}).
	\end{equation}
\end{proposition}

In order to obtain the regularity results of solution of equation \eqref{sM_equations}, we need strong assumptions on  $F$ and $B$. Namely, we assume in the rest part that
\begin{assumption}\label{assum_F}
	For an integer $\alpha\in{\mathbb N}$, $F(t,\cdot):~{\mathcal D}(M^{\alpha})\to {\mathcal D}(M^{\alpha})$ are $C^2$ functions with bounded  derivatives up to order $2$, for any $t\in[0,T]$.
\end{assumption}

\begin{assumption}\label{assum_B}
	For an integer $\beta\in{\mathbb N}$, $B(t)\in HS(U_0,{\mathcal D}(M^{\beta}))$, for any $t\in[0,T]$.
\end{assumption}

We are in the position to establish the regularity of the solution of stochastic Maxwell equations \eqref{sM_equations} in $L^{p}(\Omega;{\mathcal D}(M^k))$-norm, which is stated in the following proposition.
\begin{proposition}\label{regularity}
	Let Assumptions \ref{assum_F}-\ref{assum_B} be fulfilled with $\alpha=\beta\equiv k$, and suppose that $u_0$ is an ${\mathcal F}_0$-measurable ${\mathbb H}$-valued random variable satisfying $\|u_0\|_{L^{p}(\Omega;{\mathcal D}(M^k))}<\infty$ for some $p\geq 2$. Then the mild solution \eqref{mild sol} satisfies
	\begin{eqnarray}
	\sup_{t\in[0,T]}{\mathbb E}\|u(t)\|_{{\mathcal D}(M^k)}^{p}\leq C(1+\|u_0\|^{p}_{L^{p}(\Omega;{\mathcal D}(M^k))}),
	\end{eqnarray}
	where the positive constant $C$ may depend on the coefficients $F$ and $B$, $p$, $T$.
\end{proposition}
\begin{proof}
	The proof is similar as that of Proposition 3.1 in \cite{CHJ2018}.
\end{proof}

\begin{proposition}\label{holder}
	Under the same assumptions as in Proposition \ref{regularity}, we have for $0\leq t,s\leq T$,
	\begin{eqnarray}
&&	{\mathbb E}\|u(t)-u(s)\|_{{\mathcal D}(M^{k-1})}^p\leq C|t-s|^{p/2},\\
&&	\|{\mathbb E}(u(t)-u(s))\|_{{\mathcal D}(M^{k-1})}\leq C|t-s|,
	\end{eqnarray}
	where the positive constant $C$ may depend on $p$, $T$, 
	and $\|u_0\|_{L^{p}(\Omega;{\mathcal D}(M^k))}$.
\end{proposition}
\begin{proof}
	The proof is similar as that of Proposition 3.2 in \cite{CHJ2018}.
\end{proof}

\subsubsection{Physical properties}
In this part, we derive some physical properties of stochastic Maxwell equations \eqref{sM_equations}, including the energy evolution law and divergence evolution law.

Notice that in the deterministic case if we endow perfectly electric conducting (PEC) boundary condition ${\bf n}\times {\bf E}=0$, on $\partial D$, the Poynting theorem states the relationship satisfied by the electromagnetic energy:
\[
\partial_{t}{\mathcal H}(u(t))=2\langle u(t), F(t,u(t))\rangle_{\mathbb H},
\]
where the energy is ${\mathcal H}(u(t)):=\|u(t)\|^2_{\mathbb H}$.

Now we investigate the  energy evolution law  for stochastic Maxwell equations \eqref{sM_equations}, which is stated in the following theorem.
\begin{proposition}
		Under the same assumptions as in Proposition \ref{wellposedness thm}, we have $\forall$ $t\in[0,T]$,
	\begin{equation}\label{energy}
	\begin{split}
	\mathcal{H}(u(t))=\mathcal{H}(u_0)+\int_0^t\Big(2\langle{u(s)},F(s,u(s))\rangle_{\mathbb H}+\|B(s)\|^2_{HS(U_0,{\mathbb H})}\Big){\rm d}s+2\int_0^t\langle{u(s)},B(s)\rangle_{\mathbb H}{\rm d}W(s), ~{\mathbb P}\text{-}a.s.,
	\end{split}
	\end{equation}
	where $u$ is the solution of \eqref{sM_equations} given by Proposition \ref{wellposedness thm}.
\end{proposition}
\begin{proof}
	The proof is based on the formal application of It\^{o} formula to functional $$\mathcal{H}(u)=\|u\|^2_{\mathbb H}.$$ Since $\mathcal{H}(u)$ is Fr\'{e}chet derivable, the derivatives of $\mathcal{H}(u)$ along direction $\phi$ and $(\phi,\varphi)$ are as follows:
	\begin{equation}\label{Fre}
	D\mathcal{H}(u)(\phi)=2\langle{\mathcal{E}},\phi\rangle_{\mathbb H},\quad D^2\mathcal{H}(u)(\phi,\varphi)=2\langle\varphi,\phi\rangle_{\mathbb H}.
	\end{equation}
	From It\^{o} formula (see Theorem 4.32 in \cite{PZ2014}), we have
	\begin{equation}\label{InfIto}
	\begin{split}
	\mathcal{H}(u(t))&=\mathcal{H}(u_0)+\int_{0}^{t}\langle D\mathcal{H}(u(s)), B(s){\rm d}W(s)\rangle_{\mathbb H}\\
	&\quad+\int_{0}^{t}\langle D\mathcal{H}(u(s)), Mu(s)+F(s,u(s))\rangle_{\mathbb H}{\rm d}s\\[2mm]
	&\quad+\frac{1}{2}\int_0^t{\rm Tr}[D^2\mathcal{H}(u(s)) (B(s)Q^{\frac{1}{2}})(B(s)Q^{\frac{1}{2}})^*]{\rm d}s.
	\end{split}
	\end{equation}
	Substitute \eqref{Fre} into \eqref{InfIto} leads to
	\begin{equation*}
	\begin{split}
	\mathcal{H}(u(t))&=\mathcal{H}(u_0)+2\int_0^t\langle{u(s)},Mu(s)+F(s,u(s))\rangle_{\mathbb H}{\rm d}s\\
	&+2\int_0^t\langle{u(s)},B(s)\rangle_{\mathbb H}{\rm d}W(s)+\int_0^t\|B(s)\|^2_{HS(U_0,{\mathbb H})}{\rm d}s.
	\end{split}
	\end{equation*}
By using
	\begin{equation*}
	\langle Mu,u\rangle_{\mathbb H}=0\quad \forall~u\in {\mathcal D}(M),
	\end{equation*}	
 the proof is completed.
\end{proof}
\begin{remark}
	Compare the evolution of the averaged energy, i.e., the expectation of the equation \eqref{energy}, with the deterministic case, we found that there's one extra  term $\int_0^t \|B(s)\|^2_{HS(U_0,{\mathbb H})}{\rm d}s$ in stochastic case. That's the  effect caused by the additive noise, see also \cite[Theorem 2.1]{CHZ2016}.
\end{remark}

In the deterministic case, it is well known that the electromagnetic field is divergence free if the medium is lossless, i.e., $F=0$ in the deterministic Maxwell equation.
The following proposition sates the divergence evolution law for the stochastic Maxwell equations \eqref{sM_equations}.
	\begin{proposition}
		Under the assumptions in Proposition \ref{regularity} with $k=1$. The averaged divergence of system \eqref{sto_max} satisfies
		\begin{equation}\label{div_EH}
		\begin{split}
		\mathbb{E}({\rm div}\left(\varepsilon{\bf E}(t))\right)=&\mathbb{E}({\rm div}\left(\varepsilon{\bf E}_0)\right)-\mathbb{E}\left(\int_0^t{\rm div}{\bf J}_e{\rm d}s\right),\\[2mm]
		\mathbb{E}({\rm div}\left(\mu{\bf H}(t))\right)=&\mathbb{E}({\rm div}\left(\mu{\bf H}_0)\right)-\mathbb{E}\left(\int_0^t{\rm div}{\bf J}_m{\rm d}s\right),
		\end{split}
		\end{equation}
		where $u=({\bf E}^{T},{\bf H}^{T})^{T}$ is the solution of \eqref{sM_equations} given by Proposition \ref{wellposedness thm}.
	\end{proposition}
	\begin{proof}
		Denote $\Psi({\bf E}(t))={\rm div}(\varepsilon{\bf E}(t))$. Since $\Psi$ is Fr\'{e}chet derivable, the derivatives of $\Psi$ along direction $\phi$ or $(\phi,\varphi)$ are
		\begin{equation}
		D\Psi({\bf E})(\phi)={\rm div}(\varepsilon\phi),\quad D^2\Psi({\bf E})(\phi,\varphi)=0.
		\end{equation}
		By applying It\^{o} formula formally to $\Psi({\bf E}(t))$, it yields
		\begin{equation}\label{div_E}
		\begin{split}
		\Psi({\bf E}(t))=&\Psi({\bf E}_0)+\int_0^tD\Psi({\bf E}(s))({\rm d}{\bf E})+\frac{1}{2}\int_0^t{\rm Tr}\left[D^2\Psi({\bf E}(s))({\rm d}{\bf E},{\rm d}{\bf E})\right]\\[2mm]
		=&\Psi({\bf E}_0)-\int_0^t{\rm div}\left({\bf J}_e{\rm d}s+{\bf J}_e^r{\rm d}W(s)\right)+\int_0^t{\rm div}\left(\nabla\times{\bf H}\right){\rm d}s\\[2mm]
		=&\Psi({\bf E}_0)-\int_0^t{\rm div}{\bf J}_e{\rm d}s-\int_0^t{\rm div}\left({\bf J}_e^r{\rm d}W(s)\right),
		\end{split}
		\end{equation}
		where the last equality is due to $\nabla\cdot(\nabla\times {\bf \psi})=0,~\forall~{\bf \psi}({\bf x})\in\mathbb{R}^3$. In the similar manner, by applying It\^{o} formula to functional $\Psi({\bf H}(t))={\rm div}(\mu{\bf H}(t))$, we can get
		\begin{equation}\label{div_H}
		\begin{split}
		\Psi({\bf H}(t))=\Psi({\bf H}_0)-\int_0^t{\rm div}{\bf J}_m{\rm d}s-\int_0^t{\rm div}\left({\bf J}_m^r{\rm d}W(s)\right).
		\end{split}
		\end{equation}
		
		The results \eqref{div_EH} follows from taking the expectation on both sides of \eqref{div_E} and \eqref{div_H}, respectively.		
		The proof is thus completed.
	\end{proof}

\begin{remark}
If the medium is lossless, i.e., $F=0$, or functions ${\bf J}_e$, ${\bf J}_m$ are divergence-free, the averaged divergence holds
\begin{equation*}
\begin{split}
\mathbb{E}({\rm div}\left(\varepsilon{\bf E}(t))\right)=\mathbb{E}({\rm div}\left(\varepsilon{\bf E}_0)\right),~~
\mathbb{E}({\rm div}\left(\mu{\bf H}(t))\right)=\mathbb{E}({\rm div}\left(\mu{\bf H}_0)\right).
\end{split}
\end{equation*}
\end{remark}

\section{Symplecticity of stochastic Maxwell equations}
In \cite{CH2016}, authors introduced the general form of infinite-dimensional stochastic  Hamiltonian system  based on a stochastic version of variation principle, and showed that the phase flow preserves the stochastic symplecticity on phase space. 
In this section, we  consider the corresponding infinite-dimensional stochastic Hamiltonian system form of stochastic Maxwell equations \eqref{sto_max}. In the sequel, we assume that $\varepsilon$ and $\mu$ are two positive constants in order to obtain the symplecticity.

We rewrite stochastic Maxwell equations \eqref{sto_max} as
\begin{equation}\label{sto_max_1}
\begin{cases}
{\rm d}{\bf E}-\varepsilon^{-1}\nabla\times {\bf H}{\rm d}t=-\varepsilon^{-1}{\bf J}_{e}(t,x,{\bf E},{\bf H}){\rm d}t-\varepsilon^{-1}{\bf J}_e^{r}(t,x)\circ{\rm d}W(t),~ &(t,{\bf x})\in(0,~T]\times D,\\[2mm]
{\rm d}{\bf H}+\mu^{-1}\nabla\times {\bf E}{\rm d}t=-\mu^{-1}{\bf J}_{m}(t,x,{\bf E},{\bf H}){\rm d}t-\mu^{-1}{\bf J}_m^{r}(t,x)\circ{\rm d}W(t),~ &(t,{\bf x})\in(0,~T]\times D.
\end{cases}
\end{equation}
Denote $G:[0,~T]\times L^{2}(D)^6\to L^{2}(D)^6$ a Nemytskij operator associated to ${\bf J}_{e}$, ${\bf J}_m$, which is defined by
\begin{equation}\label{F}
G(t,u(t))({\bf x})=\left( \begin{array}{c}
\mu^{-1}{\bf J}_{m}(t,{\bf x},{\bf E}(t,{\bf x}),{\bf H}(t,{\bf x}))\\
-\varepsilon^{-1}{\bf J}_{e}(t,{\bf x},{\bf E}(t,{\bf x}),{\bf H}(t,{\bf x}))
\end{array} \right), ~t\in[0,T],~{\bf x}\in D,~u(t)\in{\mathbb H}.
\end{equation}
The following lemma states the integrability condition for the existence of a potential such that $G(t,u)=\frac{\delta \widetilde{\mathcal H}_{1}(t,u)}{\delta u}$, which makes the equations \eqref{sto_max_1} be an infinite-dimensional stochastic Hamiltonian system.
For simplifying presentation, let $G$ do not depend on time $t$ explicitly, since the dependence on time causes no substantial problems in the analysis but just leads to longer formulas.
\begin{lemma}
	Let $G:L^{2}(D)^6\to L^{2}(D)^6$ be G\^ateaux derivable, and  $DG(u)\in {\mathcal L}(L^{2}(D)^6;L^{2}(D)^6)$ is an symmetric operator, i.e.,
	\[
	\langle DG(u)\phi,~\psi\rangle_{L^{2}(D)^6}=\langle\phi,~DG(u)\psi\rangle_{L^{2}(D)^6},\quad \forall~\phi,\psi\in L^{2}(D)^6,
	\]
	then there exists a functional $\widetilde{\mathcal H}_{1}:L^{2}(D)^6\to {\mathbb R},$ such that $$G(u)=\frac{\delta \widetilde{\mathcal H}_{1}(u)}{\delta u},$$ i.e.,
	$\frac{\delta \widetilde{\mathcal H}_{1}}{\delta {\bf H}}=-\varepsilon^{-1}{\bf J}_{e}$ and $\frac{\delta \widetilde{\mathcal H}_{1}}{\delta {\bf E}}=\mu^{-1}{\bf J}_{m}$.
\end{lemma}

\begin{proof}
The functional  $\widetilde{\mathcal H}_{1}(u)$ can be defined as
\begin{equation}
\widetilde{\mathcal H}_{1}(u)=\int_0^1 \langle u,~G(\lambda u)\rangle_{L^{2}(D)^6}{\rm d}\lambda+C(x).
\end{equation}
 The functional derivative of $\widetilde{\mathcal H}_{1}(u)$ leads to
\begin{align*}
\delta \widetilde{\mathcal H}_{1} (u)(\phi)=&\langle \frac{\delta \widetilde{\mathcal H}_{1}(u)}{\delta u},~\phi\rangle_{L^{2}(D)^6}
=\lim_{\epsilon\to 0}\frac{1}{\epsilon}\Big[\widetilde{\mathcal H}_{1} (u+\epsilon \phi)-\widetilde{\mathcal H}_{1} (u)\Big]\\
=&\lim_{\epsilon\to 0}\frac{1}{\epsilon}\Big[ \int_0^1 \langle u+\epsilon\phi,~G(\lambda u+\epsilon\lambda\phi)\rangle_{L^{2}(D)^6}-\langle u,~G(\lambda u)\rangle_{L^{2}(D)^6}{\rm d}\lambda \Big]\\
=& \int_0^1 \langle u,~\lim_{\epsilon\to 0}\frac{1}{\epsilon}\big[G(\lambda u+\epsilon\lambda\phi)-G(\lambda u)\big]\rangle_{L^{2}(D)^6}{\rm d}\lambda
+\lim_{\epsilon\to 0}\int_0^1 \langle \phi,~G(\lambda u+\epsilon\lambda\phi)\rangle_{L^{2}(D)^6}{\rm d}\lambda,
\end{align*}
where the last step is from the Lebesgue dominated theorem and Lipschitz condition \eqref{bound partialJ}. By the definition of G\^ateaux derivative, we get
\begin{align*}
\langle \frac{\delta \widetilde{\mathcal H}_{1}(u)}{\delta u},~\phi\rangle_{L^{2}(D)^6}
&= \int_0^1 \lambda\langle u,~DG(\lambda u)\phi\rangle_{L^{2}(D)^6}{\rm d}\lambda
+\int_0^1 \langle \phi,~G(\lambda u)\rangle_{L^{2}(D)^6}{\rm d}\lambda \\
&=\langle \int_0^1\Big(\lambda DG(\lambda u)u +G(\lambda u)\Big){\rm d}\lambda,~\phi\rangle_{L^{2}(D)^6},
\end{align*}
where we have used the symmetry property of $DG(u)$. Therefore,
\[
\frac{\delta \widetilde{\mathcal H}_{1}(u)}{\delta u}=\int_0^1\Big(\lambda DG(\lambda u)u +G(\lambda u)\Big){\rm d}\lambda\\
=\int_0^1 \frac{\rm d}{\rm d\lambda}\Big(\lambda G(\lambda u)\Big){\rm d}\lambda=G(u).
\]
Thus we finish the proof.
\end{proof}

Therefore,
equations \eqref{sto_max_1} is a stochastic Hamiltonian system, whose infinite-dimensional stochastic Hamiltonian system form is given by
	\begin{equation}
	\begin{split}
	\begin{bmatrix}
	{\rm d}{\bf E}\\[2mm]
{\rm d}{\bf H}
	\end{bmatrix}
	&=\begin{bmatrix}
	0 & Id\\[2mm]
 -Id & 0
	\end{bmatrix}
	\begin{bmatrix}
	\mu^{-1}\nabla\times{\bf E}+\mu^{-1}{\bf J}_m\\[2mm]
 \varepsilon^{-1}\nabla\times{\bf H}-\varepsilon^{-1}{\bf J}_e
	\end{bmatrix}
	{\rm d}t
	+\begin{bmatrix}
	0 & Id\\[2mm]
 -Id & 0
	\end{bmatrix}
	\begin{bmatrix}
	\mu^{-1}{\bf J}_m^r \\[2mm]
 -\varepsilon^{-1}{\bf J}_e^r
	\end{bmatrix}
	{\rm d}W(t)\\[2mm]
	&={\mathbb J}\begin{bmatrix}
	\frac{\delta{\mathcal H}_1}{\delta {\bf E}} \\[2mm]
 \frac{\delta{\mathcal H}_1}{\delta {\bf H}}\end{bmatrix}
	{\rm d}t
	+{\mathbb J}\begin{bmatrix}
	\frac{\delta{\mathcal H}_2}{\delta {\bf E}} \\[2mm]
 \frac{\delta{\mathcal H}_2}{\delta {\bf H}}\end{bmatrix}
	\circ{\rm d}W(t)
	\end{split}
	\end{equation}
with the standard skew-adjoint operator ${\mathbb J}$ on $L^2(D)^6$ with standard inner product, the Hamiltonians
\[
{\mathcal H}_1=\int_{D}\frac12\Big(\mu^{-1}{\bf E}\cdot \nabla\times{\bf E}
+\varepsilon^{-1}{\bf H}\cdot\nabla\times{\bf H}\Big){\rm d}x+\widetilde{\mathcal H}_{1},
\]
and
\[
{\mathcal H}_2=\int_{D}\Big(\mu^{-1}{\bf J}_{m}^r\cdot{\bf E}-\varepsilon^{-1}{\bf J}_e^r\cdot{\bf H}\Big){\rm d}x.
\]

 For simplicity in notations, we denote ${\bf E}_0$, ${\bf H}_0$ by ${\bf e}$, ${\bf h}$, respectively.
The symplectic form for system \eqref{sto_max_1} is given by
\begin{equation}\label{symplectic structure}
\overline{\omega}(t)=\int_{D}{\rm d}{\bf E}(t,{\bf x})\wedge{\rm d}{\bf H}(t,{\bf x}){\rm d}{\bf x},
\end{equation}
where the overbar on $\omega$ is a reminder that the differential 2-form ${\rm d}{\bf E}\wedge {\rm d}{\bf H}$ is integrated over the space. Preservation of the symplectic form \eqref{symplectic structure} means that the spatial integral of the oriented areas of projections onto the coordinate planes $({\bf e},{\bf h})$ is an integral invariant. We say that the phase flow of \eqref{sto_max_1} preserves symplectic structure if and only if
\[
\frac{{\rm d}}{{\rm d}t}\overline{\omega}(t)=0.
\]

\begin{remark}
To avoid confusion, we note that the differentials in \eqref{sto_max_1} and \eqref{symplectic structure} have different meanings. In \eqref{sto_max_1}, ${\bf E}$, ${\bf H}$ are treated as functions of time, and ${\bf e}$, ${\bf h}$ are fixed parameters, while differentiation in \eqref{symplectic structure} is made with respect to the initial data ${\bf e}$, ${\bf h}$.
\end{remark}

We have  the following result on the stochastic symplecticity of stochastic Maxwell equations \eqref{sto_max_1}.
\begin{theorem}
The phase flow of stochastic Maxwell equations \eqref{sto_max_1} preserves symplectic structure:
\begin{equation}
{\omega}(t)={\omega}(0), \quad {\mathbb P}\text{-}a.s.
\end{equation}
\end{theorem}
\begin{proof}
From the formula of change of variables in differential forms, it yields
\begin{equation}\label{omega_def}
\begin{split}
\overline{\omega}(t)=&\int_{D}\Big(\frac{\partial {\bf E}}{\partial {\bf e}}{\rm d}{\bf e}+\frac{\partial{\bf E}}{\partial{\bf h}}{\rm d}{\bf h}\Big)\wedge
\Big(\frac{\partial {\bf H}}{\partial {\bf e}}{\rm d}{\bf e}+\frac{\partial{\bf H}}{\partial{\bf h}}{\rm d}{\bf h}\Big){\rm d}{\bf x}\\[1.5mm]
=&\int_{D}\Big[{\rm d}{\bf e}\wedge \Big(\frac{\partial {\bf E}}{\partial {\bf e}}\Big)^{T}\frac{\partial {\bf H}}{\partial {\bf e}}{\rm d}{\bf e}\Big]{\rm d}{\bf x}
+\int_{D}\Big[{\rm d}{\bf h}\wedge \Big(\frac{\partial {\bf E}}{\partial {\bf h}}\Big)^{T}\frac{\partial {\bf H}}{\partial {\bf h}}{\rm d}{\bf h}\Big]{\rm d}{\bf x}\\[1.5mm]
&+\int_{D}\Big[{\rm d}{\bf e}\wedge
\bigg(\Big(\frac{\partial {\bf E}}{\partial {\bf e}}\Big)^{T}
\frac{\partial{\bf H}}{\partial{\bf h}}
-\Big(\frac{\partial {\bf H}}{\partial {\bf e}}\Big)^{T}
\frac{\partial{\bf E}}{\partial{\bf h}}\bigg)
{\rm d}{\bf h} \Big]{\rm d}{\bf x}.
\end{split}
\end{equation}

We set ${\bf E}_{\bf e}=\frac{\partial {\bf E}}{\partial {\bf e}}$, ${\bf E}_{\bf h}=\frac{\partial {\bf E}}{\partial {\bf h}}$,
${\bf H}_{\bf e}=\frac{\partial {\bf H}}{\partial {\bf e}}$ and ${\bf H}_{\bf h}=\frac{\partial {\bf H}}{\partial {\bf h}}$. Now, thanks to the differentiability with respect to initial data of stochastic infinite-dimensional equations (see \cite[Chapter 9]{PZ2014}), we have
\begin{align}
&{\rm d}{\bf E}_{\bf e}=\Big(\varepsilon^{-1}\nabla\times{\bf H}_{\bf e}+\frac{\delta^2\widetilde{\mathcal H}_{1}}{\delta {\bf E}\delta {\bf H}}{\bf E}_{\bf e}+\frac{\delta^2\widetilde{\mathcal H}_{1}}{\delta {\bf H}^2}{\bf H}_{\bf e}\Big){\rm d}t,~{\bf E}_{\bf e}(0)=Id,\label{4.4}\\
&{\rm d}{\bf H}_{\bf e}=\Big(-\mu^{-1}\nabla\times{\bf E}_{\bf e}
-\frac{\delta^2\widetilde{\mathcal H}_{1}}{\delta {\bf E}^2}{\bf E}_{\bf e}-\frac{\delta^2\widetilde{\mathcal H}_{1}}{\delta {\bf E}\delta{\bf H}}{\bf H}_{\bf e}
\Big){\rm d}t,~{\bf H}_{\bf e}(0)=0,\label{4.5}\\
&{\rm d}{\bf E}_{\bf h}=\Big(\varepsilon^{-1}\nabla\times{\bf H}_{\bf h}
+\frac{\delta^2\widetilde{\mathcal H}_{1}}{\delta {\bf E}\delta {\bf H}}{\bf E}_{\bf h}+\frac{\delta^2\widetilde{\mathcal H}_{1}}{\delta {\bf H}^2}{\bf H}_{\bf h}
\Big){\rm d}t,~{\bf E}_{\bf h}(0)=0,\label{4.6}\\
&{\rm d}{\bf H}_{\bf h}=\Big(-\mu^{-1}\nabla\times{\bf E}_{\bf h}
-\frac{\delta^2\widetilde{\mathcal H}_{1}}{\delta {\bf E}^2}{\bf E}_{\bf h}-\frac{\delta^2\widetilde{\mathcal H}_{1}}{\delta {\bf E}\delta{\bf H}}{\bf H}_{\bf h}
\Big){\rm d}t,~{\bf H}_{\bf h}(0)=Id.\label{4.7}
\end{align}
	From equality \eqref{omega_def}, we get
	\begin{equation}\label{4.8}
	\begin{split}
	\frac{{\rm d}\overline{\omega}(t)}{{\rm d} t}=&
	\int_{D}\left[{\rm d}{\bf e}\wedge \frac{\rm d}{{\rm d}t}\bigg(\Big(\frac{\partial {\bf E}}{\partial {\bf e}}\Big)^{T}\frac{\partial {\bf H}}{\partial {\bf e}}\bigg){\rm d}{\bf e}
	+{\rm d}{\bf h}\wedge \frac{\rm d}{{\rm d}t}\bigg(\Big(\frac{\partial {\bf E}}{\partial {\bf h}}\Big)^{T}\frac{\partial {\bf H}}{\partial {\bf h}}\bigg){\rm d}{\bf h}\right]{\rm d}{\bf x}\\[2mm]
	&+\int_{D}\left[{\rm d}{\bf e}\wedge
	\frac{\rm d}{{\rm d}t}\bigg(\Big(\frac{\partial {\bf E}}{\partial {\bf e}}\Big)^{T}
	\frac{\partial{\bf H}}{\partial{\bf h}}
	-\Big(\frac{\partial {\bf H}}{\partial {\bf e}}\Big)^{T}
	\frac{\partial{\bf E}}{\partial{\bf h}}\bigg)
	{\rm d}{\bf h} \right]{\rm d}{\bf x}.
	\end{split}
	\end{equation}
	Substituting equations \eqref{4.4}-\eqref{4.7} into the above equality, and using the symmetric property of $\frac{\delta^2\widetilde{\mathcal H}_{1}}{\delta {\bf E}\delta {\bf H}}$, $\frac{\delta^2\widetilde{\mathcal H}_{1}}{\delta {\bf E}^2}$ and $\frac{\delta^2\widetilde{\mathcal H}_{1}}{\delta {\bf H}^2}$,
	 it holds
\begin{align*}
	\frac{{\rm d}\overline{\omega}(t)}{{\rm d} t}=&\int_{D}\Big[{\rm d}{\bf e}\wedge \bigg(\varepsilon^{-1}\big(\nabla\times{\bf H}_{\bf e}\big)^{T}{\bf H}_{\bf e}
	-\mu^{-1}{\bf E}_{\bf e}^{T}\nabla\times{\bf E}_{\bf e}\bigg){\rm d}{\bf e}\Big]{\rm d}{\bf x}\\
&+\int_{D}\Big[{\rm d}{\bf h}\wedge\bigg(\varepsilon^{-1}\big(\nabla\times{\bf H}_{\bf h}\big)^{T}{\bf H}_{\bf h}
	-\mu^{-1}{\bf E}_{\bf h}^{T}\nabla\times{\bf E}_{\bf h}\bigg){\rm d}{\bf h}\Big]{\rm d}{\bf x}\\
	&+\int_{D} \Big[{\rm d}{\bf e}\wedge
	\bigg(\varepsilon^{-1}\big(\nabla\times{\bf H}_{\bf e}\big)^{T}{\bf H}_{\bf h}
	-\mu^{-1}{\bf E}_{\bf e}^{T}\nabla\times{\bf E}_{\bf h}\bigg){\rm d}{\bf h}
	 \Big] {\rm d}{\bf x}\\
	&+\int_{D} \Big[{\rm d}{\bf e}\wedge
	\bigg(
	\mu^{-1}\big(\nabla\times{\bf E}_{\bf e}\big)^{T}{\bf E}_{\bf h}
	-\varepsilon^{-1}{\bf H}_{\bf e}^{T}\nabla\times{\bf H}_{\bf h}\bigg)
	{\rm d}{\bf h} \Big] {\rm d}{\bf x}\\
	=&\int_{D}\varepsilon^{-1}\bigg[ {\rm d}{\bf e}\wedge\big(\nabla\times{\bf H}_{\bf e}\big)^{T}{\bf H}_{\bf e}{\rm d}{\bf e}+{\rm d}{\bf h}\wedge \big(\nabla\times{\bf H}_{\bf h}\big)^{T}{\bf H}_{\bf h}{\rm d}{\bf h} \\
	&\qquad\qquad+{\rm d}{\bf e}\wedge \big(\nabla\times{\bf H}_{\bf e}\big)^{T}{\bf H}_{\bf h}{\rm d}{\bf h}
	-{\rm d}{\bf e}\wedge {\bf H}_{\bf e}^{T}\nabla\times{\bf H}_{\bf h}{\rm d}{\bf h}
	\bigg]{\rm d}{\bf x}\\
	&+\int_{D}\mu^{-1}\bigg[ {\rm d}{\bf e}\wedge\big(\nabla\times{\bf E}_{\bf e}\big)^{T}{\bf E}_{\bf e}{\rm d}{\bf e}+{\rm d}{\bf h}\wedge \big(\nabla\times{\bf E}_{\bf h}\big)^{T}{\bf E}_{\bf h}{\rm d}{\bf h} \\
	&\qquad\qquad+{\rm d}{\bf e}\wedge \big(\nabla\times{\bf E}_{\bf e}\big)^{T}{\bf E}_{\bf h}{\rm d}{\bf h}
	-{\rm d}{\bf e}\wedge {\bf E}_{\bf e}^{T}\nabla\times{\bf E}_{\bf h}{\rm d}{\bf h}
	\bigg]{\rm d}{\bf x}.
\end{align*}
The properties of wedge product lead to
\begin{align}
	\frac{{\rm d}\overline{\omega}(t)}{{\rm d} t}	=&\int_{D}\varepsilon^{-1}\bigg[ \nabla\times{\bf H}_{\bf e}{\rm d}{\bf e}\wedge{\bf H}_{\bf e}{\rm d}{\bf e}+\nabla\times{\bf H}_{\bf h}{\rm d}{\bf h}\wedge {\bf H}_{\bf h}{\rm d}{\bf h} \nonumber\\
	&\qquad\qquad+\nabla\times{\bf H}_{\bf e}{\rm d}{\bf e}\wedge {\bf H}_{\bf h}{\rm d}{\bf h}
	-{\bf H}_{\bf e}{\rm d}{\bf e}\wedge \nabla\times{\bf H}_{\bf h}{\rm d}{\bf h}
	\bigg]{\rm d}{\bf x}\nonumber\\
	&+\int_{D}\mu^{-1}\bigg[ \nabla\times{\bf E}_{\bf e}{\rm d}{\bf e}\wedge{\bf E}_{\bf e}{\rm d}{\bf e}+\nabla\times{\bf E}_{\bf h}{\rm d}{\bf h}\wedge {\bf E}_{\bf h}{\rm d}{\bf h} \nonumber\\
	&\qquad\qquad+\nabla\times{\bf E}_{\bf e}{\rm d}{\bf e}\wedge {\bf E}_{\bf h}{\rm d}{\bf h}
	-{\bf E}_{\bf e}{\rm d}{\bf e}\wedge \nabla\times{\bf E}_{\bf h}{\rm d}{\bf h}
	\bigg]{\rm d}{\bf x}\label{100}\\
	=&\int_{D}\varepsilon^{-1}\left({\rm d}\big(\nabla\times {\bf H}\big)\wedge {\rm d}{\bf H}\right)+\mu^{-1}\left({\rm d}\big(\nabla\times{\bf E}\big)\wedge{\rm d}{\bf E}\right){\rm d}{\bf x}\nonumber\\
	=&\int_{D}\varepsilon^{-1}\left(\frac{\partial}{\partial x}({\rm d}H_2\wedge {\rm d}H_3)+\frac{\partial}{\partial y}({\rm d}H_3\wedge {\rm d}H_1)+\frac{\partial}{\partial z}({\rm d}H_1\wedge {\rm d}H_2)\right){\rm d}{\bf x}\nonumber\\
	&\quad+\int_{D}\mu^{-1}\left(\frac{\partial}{\partial x}({\rm d}E_2\wedge {\rm d}E_3)+\frac{\partial}{\partial y}({\rm d}E_3\wedge {\rm d}E_1)+\frac{\partial}{\partial z}({\rm d}E_1\wedge {\rm d}E_2)\right){\rm d}{\bf x}.\nonumber
	\end{align}
From the zero boundary conditions, we derive immediately the result. Therefore the proof is completed.
\end{proof}

\section{Stochastic Runge-Kutta semidiscretizations}
In this section, we will study the stochastic Runge-Kutta semidiscretizations for stochastic Maxwell equations and state our main results.
For time interval $[0,T]$, introducing the uniform partition $0=t_0<t_1<\ldots<t_{N}=T$. Let $\tau=T/N$, and $\Delta W^{n+1}=W(t_{n+1})-W(t_n)$, $n=0,1,\ldots,N-1$. Applying $s$-stage stochastic Runge-Kutta methods, which only depend on the increments of the Wiener process, to \eqref{sM_equations} in temporal direction, we obtain
\begin{subequations}\label{RK method}
	\begin{align}
&	U_{ni}=u^{n}+\tau\sum_{j=1}^{s}a_{ij}\big(MU_{nj}+F(t_n+c_j\tau,U_{nj})\big)+\Delta W^{n+1}\sum_{j=1}^{s}\widetilde{a}_{ij}B(t_{n}+c_j\tau),\label{RK method_1}\\
&	u^{n+1}=u^{n}+\tau\sum_{i=1}^{s}b_i\big(MU_{ni}+F(t_n+c_i\tau,U_{ni})\big)+\Delta W^{n+1}\sum_{i=1}^{s}\widetilde{b}_{i}B(t_n+c_i\tau),\label{RK method_2}
	\end{align}
\end{subequations}
for $i=1,\ldots,s$ and $n=0,\ldots,N-1$. Here $A=\big(a_{ij}\big)_{s\times s}$ and $\widetilde{A}=\big(\widetilde{a}_{ij}\big)_{s\times s}$ are $s\times s$ matrices of real elements while $b^{T}=(b_1,\ldots,b_s)$ and $\widetilde{b}^{T}=(\widetilde{b}_1,\ldots,\widetilde{b}_s)$ are real vectors.

In order to prove, for a fixed $n\in\mathbb{N}$, the existence of a solution of \eqref{RK method_1}-\eqref{RK method_2}, for which the implicitness may be from the drift part,  we first introduce the concepts of algebraical stability and coercivity condition for Runge-Kutta method $(A,b)$.

\begin{definition}\label{stable}
	A Runge-Kutta method $(A,b)$ with
	$A=\big(a_{ij}\big)_{i,j=1}^{s}$ and $b=\big(b_i\big)_{i=1}^{s}$ is called algebraically stable, if $b_i\geq 0$ for $i=1,\ldots, s$ and
	\begin{equation}
	{\mathcal M}=\big(m_{ij}\big)_{i,j=1}^{s}\quad\text{with}\quad
	m_{ij}=b_ia_{ij}+b_j a_{ji}-b_{i}b_{j}
	\end{equation}
	is positive semidefinite.
\end{definition}

\begin{definition}\label{coercivity}
	We say that a Runge-Kutta matrix $A$ satisfies the coercivity condition if it is invertible, and there exists a diagonal positive definite matrix ${\mathcal K}={\rm diag}(k_i)$ and a positive scalar $\alpha$ such that
\begin{equation}\label{coercivity condition}
u^{T}{\mathcal K}(A)^{-1}u\geq \alpha u^{T}{\mathcal K}u,\quad \text{for all }u\in{\mathbb R}^{s}.
\end{equation}
\end{definition}

The coercivity plays an important role in the existence of numerical solution of Runge-Kutta method.
%
To present more clearly the stochastic Runge-Kutta methods \eqref{RK method_1}-\eqref{RK method_2}, we consider two concrete examples.

\begin{example}[Implicit Euler method]
	The implicit Euler method is an implicit stochastic Runge-Kutta method with Butcher Tableau given by
	\begin{center}
		\begin{tabular}{c|c}
			{\rm 1} & {\rm 1}\\
			\hline           & {\rm 1}
		\end{tabular}~,
		\quad\quad\begin{tabular}{c|c}
			{\rm 1} & {\rm 1}\\
			\hline           & {\rm 1}
		\end{tabular}~.
	\end{center}
If we apply the implicit Euler method to stochastic Maxwell equations \eqref{sM_equations} we obtain the recursion
\begin{align*}
&U_{n1}=u^n+\tau \Big(MU_{n1}+F(t_{n+1},U_{n1})\Big)+\Delta W^{n+1}B(t_{n+1}),\\
&u^{n+1}=u^n+\tau \Big(MU_{n1}+F(t_{n+1},U_{n1})\Big)+\Delta W^{n+1}B(t_{n+1}),
\end{align*}
where we abbreviated $t_{n+1}=t_n+\tau$. Clearly, we have $U_{n1}=u^{n+1}$ and hence we can write the
midpoint method compactly as
\begin{equation}\label{im Euler method}
u^{n+1}=u^n+\tau \Big(Mu^{n+1}+F(t_{n+1},u^{n+1})\Big)+\Delta W^{n+1}B(t_{n+1}).
\end{equation}
By introducing  operator
\begin{equation}\label{operators dis_euler}
S^{\rm IE}_{\tau}=(Id-{\tau}M)^{-1},
\end{equation}
we can write the equivalent form of implicit Euler method as
\begin{equation}\label{im Euler un}
u^{n+1}=S^{\rm IE}_{\tau}u^n+\tau S^{\rm IE}_{\tau}F^{n+1}+S^{\rm IE}_{\tau}B^{n+1}\Delta W^{n+1}.
\end{equation}
Note that the implicit Euler method is algebraical stable with ${\mathcal M}=1$, and satisfies the coercivity condition.
\end{example}

\begin{example}[Midpoint method]
	The midpoint method is another example of implicit stochastic Runge-Kutta method which is given by
\begin{center}
\begin{tabular}{c|c}
       {\rm 1/2} & {\rm 1/2}\\
\hline           & {\rm 1}
\end{tabular}~,
\quad\quad\begin{tabular}{c|c}
       {\rm 1/2} & {\rm 1/2}\\
\hline           & {\rm 1}
\end{tabular}.
\end{center}
If we apply the midpoint method to stochastic Maxwell equations \eqref{sM_equations} we obtain the recursion
	\begin{align*}
	U_{n1}=u^n+\frac{\tau}{2} \Big(MU_{n1}+F(t_{n+1/2},U_{n1})\Big)+\frac{\Delta W^{n+1}}{2}B(t_{n+1/2}),\\
	u^{n+1}=u^n+\tau \Big(MU_{n1}+F(t_{n+1/2},U_{n1})\Big)+\Delta W^{n+1}B(t_{n+1/2}),
	\end{align*}
	where we abbreviated $t_{n+1/2}=t_n+\tau/2$. Clearly, we have $U_{n1}=(u^{n+1}+u^{n})/2$ and hence we can write the
	midpoint method compactly as
	\begin{equation}\label{midpoint method}
	u^{n+1}=u^{n}+\frac{\tau}{2}M(u^{n+1}+u^{n})+\tau F^{n+\frac12}+ B^{n+\frac12}\Delta W^{n+1},
	\end{equation}
	where $F^{n+\frac12}=F(t_{n+\frac12},(u^n+u^{n+1})/2)$ and $B^{n+\frac12}=B(t_{n+\frac12})$. By introducing  operators
	\begin{equation}\label{operators dis_Mid}
	S^{\rm Mid}_{\tau}=(Id-\frac{\tau}{2}M)^{-1}(I+\frac{\tau}{2}M), ~and~ T^{\rm Mid}_{\tau}=(Id-\frac{\tau}{2}M)^{-1},
	\end{equation} we can write the equivalent form of midpoint method as
	\begin{equation}\label{mild un}
	u^{n+1}=S^{\rm Mid}_{\tau}u^n+\tau T^{\rm Mid}_{\tau}F^{n+\frac12}+T^{\rm Mid}_{\tau}B^{n+\frac12}\Delta W^{n+1}.
	\end{equation}	
	Note that the midpoint method is algebraical stable with ${\mathcal M}=0$ which means stochastic symplecticity (see Theorem \ref{symplecticity of RK}), and satisfies the coercivity condition.
\end{example}

\subsection{Symplectic condition of stochastic Runge-Kutta semidiscretizations}	
In this subsection, we analyze the condition of symplecticity for stochastic Runge-Kutta semidiscretizations \eqref{RK method_1}-\eqref{RK method_2}.
	\begin{theorem}\label{symplecticity of RK}
		Assume that the coefficients $a_{ij},b_i$ of stochastic Runge-Kutta method \eqref{RK method_1}-\eqref{RK method_2} satisfy
		\begin{align}\label{cond00}
	m_{ij}=b_ia_{ij}+b_j a_{ji}-b_{i}b_{j}\equiv 0,
		\end{align}
		for all $i,j=1,2,\cdots,s$, then the  \eqref{RK method_1}-\eqref{RK method_2} is stochastic symplectic with the discrete stochastic symplectic conservation law ${\mathbb P}$-a.s.,
		\[\bar{\omega}^{n+1}=\int_{D}{\rm d}{\bf E}^{n+1}\wedge {\rm d}{\bf H}^{n+1}{\rm d}{\bf x} =\int_{D}{\rm d}{\bf E}^{n}\wedge {\rm d}{\bf H}^{n}{\rm d}{\bf x}=\bar{\omega}^{n}.\]
		
	\end{theorem}

	\begin{proof}
		It follows from equations \eqref{RK method_1} and \eqref{RK method_2} that
		\begin{subequations}
			\begin{align}
			&{\rm d}U_{ni}={\rm d}u^n+\tau\sum_{j=1}^{s}a_{ij}M{\rm d}U_{nj}+\tau\sum_{j=1}^{s}a_{ij}{\mathbb J}\frac{\delta^2 \widetilde{\mathcal H}_{1}}{\delta u^2}{\rm d}U_{nj},\label{var_RK_1}\\
			&{\rm d}u^{n+1}={\rm d}u^{n}+\tau\sum_{i=1}^{s}b_iM{\rm d}U_{ni}+\tau\sum_{i=1}^{s}b_{i}{\mathbb J}\frac{\delta^2 \widetilde{\mathcal H}_{1}}{\delta u^2}{\rm d}U_{ni},\label{var_RK_2}
			\end{align}
		\end{subequations}
	where we use $F={\mathbb J}\frac{ \delta\widetilde{\mathcal H}_{1} }{\delta u}$.
		Therefore, we have
		\begin{align}\label{eq1}
		&{\rm d}u^{n+1}\wedge {\mathbb J}{\rm d}u^{n+1}-{\rm d}u^{n}\wedge {\mathbb J}{\rm d}u^{n}\nonumber\\
		&=\left({\rm d}u^{n}+\tau\sum_{i=1}^{s}b_iM{\rm d}U_{ni}+\tau\sum_{i=1}^{s}b_{i}{\mathbb J}\frac{\delta^2 \widetilde{\mathcal H}_{1}}{\delta u^2}{\rm d}U_{ni}\right)\nonumber\\
		&\qquad\wedge {\mathbb J}\left({\rm d}u^{n}+\tau\sum_{i=1}^{s}b_iM{\rm d}U_{ni}+\tau\sum_{i=1}^{s}b_{i}{\mathbb J}\frac{\delta^2 \widetilde{\mathcal H}_{1}}{\delta u^2}{\rm d}U_{ni}\right)-{\rm d}u^{n}\wedge {\mathbb J}{\rm d}u^{n}\nonumber\\
		&=\tau\sum_{i=1}^{s}b_i\left({\rm d}u^n\wedge{\mathbb J}M{\rm d}U_{ni}+M{\rm d}U_{ni}\wedge{\mathbb J}{\rm d}u^n\right)\\
		&\qquad+\tau\sum_{i=1}^{s}b_{i}\Big( {\rm d}u^n\wedge{\mathbb J}^2\frac{\delta^2 \widetilde{\mathcal H}_{1}}{\delta u^2}{\rm d}U_{ni}+{\mathbb J}\frac{\delta^2 \widetilde{\mathcal H}_{1}}{\delta u^2}{\rm d}U_{ni}\wedge{\mathbb J}{\rm d}u^n \Big)\nonumber\\
		&\qquad+\tau^2\sum_{i,j=1}^{s}b_ib_j\Big(M{\rm d}U_{ni}\wedge{\mathbb J}M{\rm d}U_{nj}+
		{\mathbb J}\frac{\delta^2 \widetilde{\mathcal H}_{1}}{\delta u^2}{\rm d}U_{ni}\wedge {\mathbb J}^2\frac{\delta^2 \widetilde{\mathcal H}_{1}}{\delta u^2}{\rm d}U_{nj}\Big)\nonumber\\
		&\qquad+\tau^2\sum_{i,j=1}^{s}b_ib_j\Big(
		M{\rm d}U_{ni}\wedge {\mathbb J}^2\frac{\delta^2 \widetilde{\mathcal H}_{1}}{\delta u^2}{\rm d}U_{nj}+{\mathbb J}\frac{\delta^2 \widetilde{\mathcal H}_{1}}{\delta u^2}{\rm d}U_{ni}\wedge {\mathbb J}M{\rm d}U_{nj}\nonumber
		\Big).
		\end{align}
		From \eqref{var_RK_1}, we have
		\begin{equation*}
		{\rm d}u^n={\rm d}U_{ni}-\tau\sum_{j=1}^{s}a_{ij}M{\rm d}U_{nj}-\tau\sum_{j=1}^{s}a_{ij}{\mathbb J}\frac{\delta^2 \widetilde{\mathcal H}_{1}}{\delta u^2}{\rm d}U_{nj}.
		\end{equation*}		
		Substituting  the above equation into the first and second terms on the right-hand side of \eqref{eq1}, we obtain
		\begin{equation}\label{eq2}
		\begin{split}
		&{\rm d}u^{n+1}\wedge {\mathbb J}{\rm d}u^{n+1}-{\rm d}u^{n}\wedge {\mathbb J}{\rm d}u^{n}\\
		&=\tau\sum_{i=1}^{s}b_i\left({\rm d}U_{ni}\wedge{\mathbb J}M{\rm d}U_{ni}+M{\rm d}U_{ni}\wedge{\mathbb J}{\rm d}U_{ni}\right)\\
		&\qquad+\tau\sum_{i=1}^{s}b_i\left({\rm d}U_{ni}\wedge{\mathbb J}^2\frac{\delta^2 \widetilde{\mathcal H}_{1}}{\delta u^2}{\rm d}U_{ni}+{\mathbb J}\frac{\delta^2 \widetilde{\mathcal H}_{1}}{\delta u^2}{\rm d}U_{ni}\wedge{\mathbb J}{\rm d}U_{ni}\right)\\
		 &\qquad+\tau^2\sum_{i,j=1}^{s}\left(b_ib_j-b_ia_{ij}-b_ja_{ji}\right)\left(M{\rm d}U_{ni}\wedge{\mathbb J}M{\rm d}U_{nj}\right)\\
		 &\qquad+2\tau^2\sum_{i,j=1}^{s}\left(b_ib_j-b_ia_{ij}-b_ja_{ji}\right)\left(M{\rm d}U_{ni}\wedge{\mathbb J}^2\frac{\delta^2 \widetilde{\mathcal H}_{1}}{\delta u^2}{\rm d}U_{nj}\right)\\
		 &\qquad+\tau^2\sum_{i,j=1}^{s}\left(b_ib_j-b_ia_{ij}-b_ja_{ji}\right)\left({\mathbb J}\frac{\delta^2 \widetilde{\mathcal H}_{1}}{\delta u^2}{\rm d}U_{ni}\wedge{\mathbb J}^2\frac{\delta^2 \widetilde{\mathcal H}_{1}}{\delta u^2}{\rm d}U_{nj}\right).
		\end{split}
		\end{equation}
		From the symmetry of $\frac{\delta^2 \widetilde{\mathcal H}_{1}}{\delta u^2}$, the value of the second term on the right-hand side of \eqref{eq2} is zero. From the  symplectic condition \eqref{cond00}, the third, forth and fifth terms on the right-hand side of \eqref{eq2} are also zeros. Therefore,
		\begin{equation*}
		\begin{split}
			{\rm d}u^{n+1}\wedge {\mathbb J}{\rm d}u^{n+1}-{\rm d}u^{n}\wedge {\mathbb J}{\rm d}u^{n}
		=\tau\sum_{i=1}^{s}b_i\left({\rm d}U_{ni}\wedge{\mathbb J}M{\rm d}U_{ni}+M{\rm d}U_{ni}\wedge{\mathbb J}{\rm d}U_{ni}\right).
		\end{split}
		\end{equation*}
 Recalling $u=\begin{pmatrix}{\bf E}\\ {\bf H}\end{pmatrix}$ and the Maxwell operator $M$ in \eqref{M_operator}, and using the skew-symmetry of ${\mathbb J}$, it yields
		\begin{align}
		&{\rm d}{\bf E}^{n+1}\wedge {\rm d}{\bf H}^{n+1}-{\rm d}{\bf E}^{n}\wedge {\rm d}{\bf H}^{n}\nonumber\\
		&=\frac{1}{2}\left({\rm d}u^{n+1}\wedge {\mathbb J}{\rm d}u^{n+1}-{\rm d}u^{n}\wedge {\mathbb J}{\rm d}u^{n}\right)\\
		&=\tau\sum_{i=1}^{s}b_i\left({\rm d}U_{ni}\wedge{\mathbb J}M{\rm d}U_{ni}\right)\nonumber\\
		&=-\tau\sum_{i=1}^{s}b_i\left[\mu^{-1}{\rm d}{\bf E}_{ni}\wedge(\nabla\times{\rm d}{\bf E}_{ni})+\varepsilon^{-1}{\rm d}{\bf H}_{ni}\wedge(\nabla\times{\rm d}{\bf H}_{ni})\right].\nonumber
		\end{align}
		Thereby, by using the similar proof approach in the last two steps of \eqref{100} it holds
		\begin{equation*}
		\begin{split}
		&\int_{D}{\rm d}{\bf E}^{n+1}\wedge {\rm d}{\bf H}^{n+1}{\rm d}{\bf x}-\int_{D}{\rm d}{\bf E}^{n}\wedge {\rm d}{\bf H}^{n}{\rm d}{\bf x}\\
		&=-\tau\sum_{i=1}^{s}b_i\int_{D}\left[\mu^{-1}{\rm d}{\bf E}_{ni}\wedge(\nabla\times{\rm d}{\bf E}_{ni})+\varepsilon^{-1}{\rm d}{\bf H}_{ni}\wedge(\nabla\times{\rm d}{\bf H}_{ni})\right]{\rm d}{\bf x}
		=0.
		\end{split}
		\end{equation*}		
		Thus, the proof is completed.
	\end{proof}

\begin{remark}
	Note that for a symplectic Runge-Kutta method, it satisfies algebraically stable condition automatically. 
\end{remark}

\subsection{Regularity of stochastic Runge-Kutta semidiscretizations}
In this subsection, we present the results of well-posedness and regularity of numerical solution given by stochastic Runge-Kutta method \eqref{RK method_1}-\eqref{RK method_2} satisfying the algebraical stability and coercivity conditions.

First, we utilize Kronecker product to rewrite \eqref{RK method_1}-\eqref{RK method_2} in a compact form,
\begin{subequations}\label{RK_compact}
	\begin{align}
	&U_{n}={\bf 1}_{s}\otimes u^n+\tau \big(A\otimes M\big)U_{n}+\tau \big(A\otimes I\big)F^{n}(U_n)+\big(\widetilde{A}\otimes I\big)B^{n}\Delta W^{n+1},\label{RK_compact_1}\\[2mm]
	&u^{n+1}=u^{n}+\tau\big(b^{T}\otimes M\big)U_n+\tau\big(b^{T}\otimes I\big)F^{n}(U_n)+\big(\widetilde{b}^{T}\otimes I\big)B^{n}\Delta W^{n+1},\label{RK_compact_2}
	\end{align}
\end{subequations}
where ${\bf 1}_{s}=[1,\ldots,1]^{T}$, $I$ is the identity matrix of size $6\times 6$, and
\[
U_{n}=\begin{bmatrix}
U_{n1}\\ U_{n2}\\ \cdots \\U_{ns}
\end{bmatrix},
\qquad
F^{n}(U_n)=\begin{bmatrix}
F(t_n+c_1\tau,U_{n1})\\F(t_n+c_2\tau,U_{n2})\\ \cdots \\F(t_n+c_s\tau,U_{ns})
\end{bmatrix},
\qquad
B^{n}=\begin{bmatrix}
B(t_n+c_1\tau)\\B(t_n+c_2\tau)\\ \cdots \\B(t_n+c_s\tau)
\end{bmatrix}.
\]

Next, we give some useful estimates on the operator $(A\otimes M)$, under the coercivity condition of  matrix $A$.
\begin{lemma}\label{est operator}
Let matrix $A$ satisfy coercivity condition \eqref{coercivity condition}. Then there exists constant $C$ such that
	\begin{itemize}
		\item [(i)]$\|\Big(I_{6s\times 6s}-\tau(A\otimes M)\Big)^{-1}\|_{{\mathcal L}({\mathbb H}^s;{\mathbb H}^s)}\leq C$;
		\item [(ii)] $\|I_{6s\times 6s}-\Big(I_{6s\times 6s}-\tau(A\otimes M)\Big)^{-1}\|_{{\mathcal L}(({\mathcal D}(M))^s;{\mathbb H}^s)}\leq C\tau$.
	\end{itemize}
\end{lemma}
\begin{proof}
	In order to estimate the operator $I_{6s\times 6s}-\Big(I_{ 6s\times 6s}-\tau\big(A\otimes M\big)\Big)^{-1}$, we denote $v^{n+1}=\Big(I_{ 6s\times 6s}-\tau\big(A\otimes M\big)\Big)^{-1}v^n$, and then $\{v^n\}_{n\in{\mathbb N}}$ is the discrete solution of the following discrete system
	\begin{equation}\label{5.15}
	v^{n+1}=v^{n}+\tau\big(A\otimes M\big)v^{n+1}.
	\end{equation}
	Suppose that $A$ satisfies the coercivity  condition, we apply $\langle v^{n+1},\big({\mathcal K}A^{-1}\otimes I\big)\cdot \rangle_{{\mathbb H}^s}$ to both sides of \eqref{5.15} and get
	\begin{equation}\label{5.17}
	\begin{split}
	\langle v^{n+1},\big({\mathcal K}A^{-1}\otimes I\big)v^{n+1} \rangle_{{\mathbb H}^s}
	=&\langle v^{n+1},\big({\mathcal K}A^{-1}\otimes I\big)v^{n} \rangle_{{\mathbb H}^s}\\
	&+\tau\langle v^{n+1},\big({\mathcal K}A^{-1}\otimes I\big)\big(A\otimes M\big)v^{n+1} \rangle_{{\mathbb H}^s}.
	\end{split}
	\end{equation}
	Since
	\[
	\langle v^{n+1},\big({\mathcal K}A^{-1}\otimes I\big)v^{n+1} \rangle_{{\mathbb H}^s}
	\geq \alpha \sum_{i=1}^{s}k_i\|v^{n+1,i}\|^2_{\mathbb H}\geq \alpha\min\{k_i\} \|v^{n+1}\|^2_{{\mathbb H}^s}:=\tilde{\alpha}\|v^{n+1}\|^2_{{\mathbb H}^s},
	\]
	and
	\[
	\langle v^{n+1},\big({\mathcal K}A^{-1}\otimes I\big)\big(A\otimes M\big)v^{n+1} \rangle_{{\mathbb H}^s}=\langle v^{n+1},\big({\mathcal K}\otimes M\big)v^{n+1} \rangle_{{\mathbb H}^s}=\sum_{i=1}^{s}k_i\langle v^{n+1,i} , Mv^{n+1,i}\rangle_{{\mathbb H}}=0,
	\]
	we get for \eqref{5.17}
	\[
	\tilde{\alpha}\|v^{n+1}\|^2_{{\mathbb H}^s}\leq \langle v^{n+1},\big({\mathcal K}A^{-1}\otimes I\big)v^{n} \rangle_{{\mathbb H}^s}
	\leq \gamma \|v^{n+1}\|^2_{{\mathbb H}^s}+\frac{C}{\gamma}\|v^n\|^2_{{\mathbb H}^s},
	\]
	where $C$ depends on $|{\mathcal K}|$ and $|A^{-1}|$.
	Taking $\gamma=\tilde{\alpha}/2$ leads to
	\[
	\|v^{n+1}\|^2_{{\mathbb H}^s}\leq C\|v^{n}\|^2_{{\mathbb H}^s},
	\]
	where the constant $C$ depends on $\tilde{\alpha}$, $|{\mathcal K}|$ and $|A^{-1}|$. It means that
	\begin{equation}\label{5.18}
	\|\Big(I_{ 6s\times 6s}-\tau\big(A\otimes M\big)\Big)^{-1}v^{n}\|^2_{{\mathbb H}^s}\leq C\|v^{n}\|^2_{{\mathbb H}^s}
	\end{equation}
	Thus we show the first assertion.
	Similarly, we may show that
	\[
	\|\big(A\otimes M\big)v^{n+1}\|^2_{{\mathbb H}^s}\leq C\|\big(A\otimes M\big)v^{n}\|^2_{{\mathbb H}^s}.
	\]
	From
	\begin{equation}\label{5.16}
	\bigg[\Big(I_{ 6s\times 6s}-\tau\big(A\otimes M\big)\Big)^{-1}-I_{6s\times 6s}\bigg]v^n=v^{n+1}-v^n=\tau(A\otimes M)v^{n+1},
	\end{equation}
	it follows that
	\[
	\left\|\bigg[\Big(I_{ 6s\times 6s}-\tau\big(A\otimes M\big)\Big)^{-1}-I_{6s\times 6s}\bigg]v^n\right\|_{{\mathbb H}^s}
	=\tau\|(A\otimes M)v^{n+1}\|_{{\mathbb H}^s}
	\leq C\tau\|(A\otimes M)v^{n}\|_{{\mathbb H}^s},
	\]
	which leads to the second assertion.
	
\end{proof}

Now we are in the position to present the existence and uniqueness of the numerical solution given by the stochastic Runge-Kutta method \eqref{RK method}.

\begin{theorem}\label{wellposedness-RK}
	In addition to conditions of Proposition \ref{wellposedness thm}, let $B(t)\in HS(U_0,{\mathcal D}(M))$ for any $t\in[0,T]$.
	Let the Runge-Kutta method $(A,b)$ be algebraically stable and coercive.
			For $p\geq 2$ and fix $T=t_{N}>0$, there exists an unique ${\mathbb H}$-valued $\{{\mathcal F}_{t_n}\}_{0\leq n\leq N}$-adapted discrete solution $\{ u^n; ~n=0,1,\ldots,N\}$ of the  scheme \eqref{RK method} for sufficiently small $\tau\leq \tau^{*}$ with $\tau^{*}:= \tau^{*}(\|u_0\|_{\mathbb H},T)$, and a constant $C:=C(p,T,\sup_{t\in[0,T]}\|B(t)\|_{HS(U,{\mathcal D}(M))})>0$ such that
	\begin{align}
	\max_{1\leq i\leq s}{\mathbb E}\|U_{ni}\|_{\mathbb H}^{p}&\leq C\big({\mathbb E}\|u^n\|_{\mathbb H}^{p}+\tau\big),\label{bound U_ni}\\[0.6em]
	\max_{1\leq n\leq N}{\mathbb E}\|u^n\|^{p}_{\mathbb H}&\leq C\big(1+\|u_0\|^{p}_{L^{p}(\Omega;\mathbb H}\big).\label{bound un_RK}
	\end{align}
\end{theorem}

	\begin{proof}		
	We only present the proof for $p=2$ here, since the proof for general $p>2$ is similar.
	
	{\em Step 1: Existence and $\{{\mathcal F}_{t_n}\}_{0\leq n\leq N}$-adaptedness.} Fix a set $\Omega^{'}\subset\Omega$, ${\mathbb P}(\Omega^{'})=1$ such that $W(t,\omega)\in U$ for all $t\in[0,T]$ and $\omega\in\Omega^{'}$. In the following, let us assume that $\omega\in\Omega^{'}$. The existence of iterates $\{ u^n; ~n=0,1,\ldots,N\}$
	follows from a standard Galerkin method and Brouwer's theorem, in combining with assertions \eqref{bound U_ni}-\eqref{bound un_RK}.
	
	Define a map
	\begin{equation*}
	\begin{split}
	\Lambda:~{\mathbb H}\times U\to{\mathcal P}({\mathbb H}),
	\quad(u^n,\Delta W^{n+1})\to \Lambda(u^n,\Delta W^{n+1}),
	\end{split}
	\end{equation*}
	where ${\mathcal P}({\mathbb H})$ denotes the set of all subsets of ${\mathbb H}$, and $\Lambda(u^n,\Delta W^{n+1})$ is the set of solutions $u^{n+1}$ of \eqref{RK method}. By the closedness of the graph of $\Lambda$ and a selector theorem,  there exists a universally and Borel measurable mapping $\lambda_n:~{\mathbb H}\times U\to {\mathbb H}$ such that $\lambda_n(s_1,s_2)\in\Lambda(s_1,s_2)$ for all $(s_1,s_2)\in {\mathbb H}\times U$. Therefore, ${\mathcal F}_{t_{n+1}}$-measurability of $u^{n+1}$ follows from the Doob-Dynkin lemma.
	
	{\em Step 2: proof for \eqref{bound U_ni}.}
	From the compact formula \eqref{RK_compact_1} and the invertibility of $A$, we get
	\begin{equation}\label{Un}
	\begin{split}
	U_{n}=&\Big(I_{ 6s\times 6s}-\tau\big(A\otimes M\big)\Big)^{-1}\big({\bf 1}_{s}\otimes u^n\big)
	+\tau\Big(I_{ 6s\times 6s}-\tau\big(A\otimes M\big)\Big)^{-1}\big(A\otimes I\big)F^{n}\\
	&+\Big(I_{ 6s\times 6s}-\tau\big(A\otimes M\big)\Big)^{-1}\Big(\big(\widetilde{A}\otimes I\big)B^{n}\Delta W^{n+1}\Big).
	\end{split}
	\end{equation}
	Using assertion (i) of Lemma \ref{est operator},
	we obtain,
	\begin{equation}\label{5.7}
	\begin{split}
	\|U_{n}\|^2_{{\mathbb H}^s}
	&\leq C\|{\bf 1}_{s}\otimes u^n 
	+\tau\big(A\otimes I\big)F^{n}+\big(\widetilde{A}\otimes I\big)B^{n}\Delta W^{n+1}\|^2_{{\mathbb H}^s}\\
	&\leq C\|u^n\|^2_{{\mathbb H}}+\tau^2\sum_{i=1}^{s}\|F^{ni}\|^2_{{\mathbb H}}+\sum_{i=1}^{s}\|B^{ni}\Delta W^{n+1}\|^2_{{\mathbb H}}\\
	&\leq C\|u^n\|^2_{{\mathbb H}}+C\tau^2\sum_{i=1}^{s}\big(1+\|U_{ni}\|^2_{{\mathbb H}}\big)+\sum_{i=1}^{s}\|B^{ni}\Delta W^{n+1}\|^2_{{\mathbb H}}\\
	&	\leq C\|u^n\|^2_{{\mathbb H}}+C\tau^2+C\tau^2\|U_{n}\|^2_{{\mathbb H}^s}+\sum_{i=1}^{s}\|B^{ni}\Delta W^{n+1}\|^2_{{\mathbb H}}.
	\end{split}
	\end{equation}
	Taking expectation on both sides of \eqref{5.7}, we have
	\begin{equation}
	{\mathbb E}\|U_{n}\|^2_{{\mathbb H}^s}\leq
	C{\mathbb E}	\|u^n\|^2_{{\mathbb H}}+C\tau+C\tau^2{\mathbb E}\|U_{n}\|^2_{{\mathbb H}^s}.
	\end{equation}
	For sufficiently small step size, by Gronwall inequality,
	one gets
	\[
	{\mathbb E}\|U_{n}\|^2_{{\mathbb H}^s}\leq
	C{\mathbb E}	\|u^n\|^2_{{\mathbb H}}+C\tau.
	\]
	Because of the identity $\sum_{i=1}^{s}\|U_{ni}\|^2_{{\mathbb H}}=\|U_{n}\|^2_{{\mathbb H}^s}$,	
	the proof of \eqref{bound U_ni} is completed.
	
	{\em Step 3: Uniqueness.} The uniqueness of discrete solution follows from the uniqueness of $U_{ni}$, $i=1,\ldots,s$.	
	
	Assume that there are two different solutions $U_{n}$ and $V_{n}$ satisfying \eqref{RK_compact_1}, then it follows
	\begin{equation}
	U_{n}-V_{n}=\tau(A\otimes M)\big(U_n-V_n\big)+\tau(A\otimes I)\big(F^n(U_n)-F^n(V_n)\big),
	\end{equation}
	which is equivalent to
	\begin{equation}
	U_{n}-V_{n}=\tau\Big(I_{6s\times 6s}-\tau(A\otimes M)\Big)^{-1}(A\otimes I)\big(F^n(U_n)-F^n(V_n)\big).
	\end{equation}
	From the assertion (i) of Lemma \ref{est operator} and globally Lipschitz property of function $F$, it follows that
	\begin{equation}
	\|U_{n}-V_{n}\|_{{\mathbb H}^s}\leq C\tau \|U_{n}-V_{n}\|_{{\mathbb H}^s}.
	\end{equation}
	Obviously, when the time step $\tau$ is sufficiently small, the internal stages $U_{ni}$ is unique, hence the discrete solution $u^{n+1}$ is unique.
	
	{\em Step 4: proof for \eqref{bound un_RK}.}
	We start from \eqref{RK method_2} to get
	\begin{equation}\label{5.9}
	\begin{split}
	\|u^{n+1}\|^2_{{\mathbb H}}=&\|u^{n}\|^2_{{\mathbb H}}+\|\tau\sum_{i=1}^{s}b_i\big(MU_{ni}+F^{ni}\big)\|^2_{{\mathbb H}}
	+ \|\sum_{i=1}^{s}\widetilde{b}_iB^{ni}\Delta W^{n+1}\|^2_{{\mathbb H}}\\
	&	+2\langle u^{n},~\tau\sum_{i=1}^{s}b_i\big(MU_{ni}+F^{ni}\big) \rangle_{{\mathbb H}}			+2\langle u^{n},~\sum_{i=1}^{s}\widetilde{b}_iB^{ni}\Delta W^{n+1} \rangle_{{\mathbb H}}\\
	&+2\langle \tau\sum_{i=1}^{s}b_i\big(MU_{ni}+F^{ni}\big),~ \sum_{i=1}^{s}\widetilde{b}_iB^{ni}\Delta W^{n+1}\rangle_{{\mathbb H}}.
	\end{split}
	\end{equation}			
	From \eqref{RK method_1}, we know that
	\begin{equation}\label{5.10}
	u^{n}=U_{ni}-\tau\sum_{j=1}^{s}a_{ij}\big(MU_{nj}+F^{nj}\big)-\sum_{j=1}^{s}\widetilde{a}_{ij}B^{nj}\Delta W^{n+1},
	\end{equation}
	and then substitute \eqref{5.10} into the first term of the second line on the right-hand side of \eqref{5.9} to get
	\begin{align*}
	2\tau&\sum_{i=1}^{s}b_i	\langle u^{n},~MU_{ni}+F^{ni} \rangle_{{\mathbb H}}\\
	=&
	2\tau\sum_{i=1}^{s}b_i	\langle U_{ni}, ~ MU_{ni}+F^{ni} \rangle_{{\mathbb H}} 	-2\tau^2\sum_{i,j=1}^{s}b_ia_{ij}\langle MU_{nj}+F^{nj},~ MU_{ni}+F^{ni} \rangle_{{\mathbb H}}
	\\
	&-2\tau\sum_{i,j=1}^{s}b_i \widetilde{a}_{ij}\langle B^{nj}\Delta W^{n+1},~
	MU_{ni}+F^{ni} \rangle_{{\mathbb H}} \\
	=&2\tau\sum_{i=1}^{s}b_i	\langle U_{ni}, ~F^{ni} \rangle_{{\mathbb H}}  -\tau^2\sum_{i,j=1}^{s}\big(b_ia_{ij}+b_ja_{ji}\big)\langle MU_{nj}+F^{nj},~ MU_{ni}+F^{ni} \rangle_{{\mathbb H}}
	\\
	&-2\tau\sum_{i,j=1}^{s}b_i \widetilde{a}_{ij}\langle B^{nj}\Delta W^{n+1},~
	MU_{ni}+F^{ni} \rangle_{{\mathbb H}}
	\end{align*}
	where in the last step we have used the fact $\langle U_{ni},~MU_{ni} \rangle_{{\mathbb H}}=0 $.
	Combining the above equality together with \eqref{5.9},	 we get
	\begin{align}\label{5.3}
	\|u^{n+1}\|^2_{{\mathbb H}}=&\|u^{n}\|^2_{{\mathbb H}}
	+ \|\sum_{i=1}^{s}\widetilde{b}_iB^{ni}\Delta W^{n+1}\|^2_{{\mathbb H}}
	+2\tau\sum_{i=1}^{s}b_i	\langle U_{ni}, ~F^{ni} \rangle_{{\mathbb H}} 		\nonumber\\
	& +\tau^2\sum_{i,j=1}^{s}\big(b_ib_j-b_ia_{ij}-b_ja_{ji}\big)\langle MU_{nj}+F^{nj},~ MU_{ni}+F^{ni} \rangle_{{\mathbb H}}
	\\
	&+2\langle u^{n},~\sum_{i=1}^{s}\widetilde{b}_iB^{ni}\Delta W^{n+1} \rangle_{{\mathbb H}}+2\tau\sum_{i,j=1}^{s}\big(b_i\widetilde{b}_j-b_i \widetilde{a}_{ij}\big)\langle B^{nj}\Delta W^{n+1},~
	MU_{ni}+F^{ni} \rangle_{{\mathbb H}}.
	\nonumber
	\end{align}
	Since the method $(A,b)$ is algebraically stable, the second line of \eqref{5.3} is not positive, then we end up with
	\begin{equation}\label{5.4}
	\begin{split}
	\|u^{n+1}\|^2_{{\mathbb H}}\leq&\|u^{n}\|^2_{{\mathbb H}}
	+ \|\sum_{i=1}^{s}\widetilde{b}_iB^{ni}\Delta W^{n+1}\|^2_{{\mathbb H}}
	+2\tau\sum_{i=1}^{s}b_i	\langle U_{ni}, ~F^{ni} \rangle_{{\mathbb H}} 		\\
	&+2\langle u^{n},~\sum_{i=1}^{s}\widetilde{b}_iB^{ni}\Delta W^{n+1} \rangle_{{\mathbb H}}+2\tau\sum_{i,j=1}^{s}\big(b_i\widetilde{b}_j-b_i \widetilde{a}_{ij}\big)\langle B^{nj}\Delta W^{n+1},~
	MU_{ni}+F^{ni} \rangle_{{\mathbb H}}\\
	\leq & \|u^{n}\|^2_{{\mathbb H}} +C(1+\tau)\sum_{i=1}^{s}\|B^{ni}\Delta W^{n+1}\|^2_{{\mathbb H}}
	+C\tau\sum_{i=1}^{s}\|M(B^{ni}\Delta W^{n+1})\|^2_{{\mathbb H}}\\
	&+C\tau\sum_{i=1}^{s}\|U_{ni}\|^2_{{\mathbb H}}
	+C\tau\sum_{i=1}^{s}\|F^{ni}\|^2_{{\mathbb H}}+2C\tau\sum_{i=1}^{s}b_i	\langle U_{ni}, ~F^{ni} \rangle_{{\mathbb H}}.
	\end{split}
	\end{equation}
	Applying expectation and using conditions on $F$, $B$ and $Q$ lead to
	\begin{equation}
	{\mathbb E}\|u^{n+1}\|^2_{{\mathbb H}}\leq
	{\mathbb E}\|u^{n}\|^2_{{\mathbb H}}+C\tau+C\tau{\mathbb E}\|U_n\|^2_{{\mathbb H}^{s}}.
	\end{equation}	
	Substituting \eqref{bound U_ni} into the above inequality, we get
	\begin{equation}
	{\mathbb E}\|u^{n+1}\|^2_{{\mathbb H}} \leq (1+C\tau){\mathbb E}\|u^{n}\|^2_{{\mathbb H}}+C\tau,
	\end{equation}	
	which by Gronwall's inequality means the boundedness of numerical solution.
	Therefore we complete the proof of \eqref{bound un_RK}.
	Combining Steps 1-4, we complete the proof.
\end{proof}

\begin{remark}
Note that for the well-posedness of stochastic Runge-Kutta method, we require  the additional spatial smooth assumptions on function  $B$, which comes from term $\|M(B^{ni}\Delta W^{n+1})\|^2_{{\mathbb H}}$ and needs $\sup_{t\in[0,T]}\|B(t)\|_{HS(U_0,{\mathcal D}(M))}<\infty$.
\end{remark}

Now we are in the position to discuss the regularity in ${\mathcal D}(M^k)$ ($k\in{\mathbb N}$) of the numerical solution given by stochastic Runge-Kutta method.

	\begin{proposition}\label{prop}
Let Assumption \ref{assum_F} and Assumption \ref{assum_B} be fulfilled with $\alpha=k$ and $\beta=k+1$, respectively, and suppose the initial data $u_0\in L^{p}(\Omega;{\mathcal D}(M^k))$ for some $p\geq 2$.
	For the solution of \eqref{RK method_1}-\eqref{RK method_2},  there exists a constant $C:=C(p,T,\sup_{t\in[0,T]}\|B(t)\|_{HS(U,{\mathcal D}(M^{k+1}))})>0$ such that
		\begin{align}
\max_{1\leq i\leq s}{\mathbb E}\|U_{ni}\|_{{\mathcal D}(M^k)}^{p}&\leq C\big({\mathbb E}\|u^n\|_{{\mathcal D}(M^k)}^{p}+\tau\big),\label{bound U_ni_re}\\[0.6em]
\max_{1\leq n\leq N}{\mathbb E}\|u^n\|^{p}_{{\mathcal D}(M^k)}&\leq C\big(1+\|u_0\|^{p}_{L^{p}(\Omega;{\mathcal D}(M^k))}\big).\label{bound un_RK_re}
\end{align}
	\end{proposition}
\begin{proof}
The proof is similar as in Step2 and Step 4 of Theorem \ref{wellposedness-RK}.	
\end{proof}

\begin{proposition}\label{holder_RK}
	Under the same assumption as in Proposition \ref{prop}, we have for $0\leq t,s\leq T$,
	\begin{align}
&	{\mathbb E}\|u^{n+1}-u^{n}\|_{{\mathcal D}(M^{k-1})}^p\leq C\tau^{p/2},\\
&	\|{\mathbb E}(u^{n+1}-u^{n})\|_{{\mathcal D}(M^{k-1})}\leq C\tau.
	\end{align}
	Moreover, if $u^{n+1}$ is replaced by $U_{ni}$, the above estimates still hold.
\end{proposition}

\subsection{Error analysis of stochastic Runge-Kutta semidiscretizations}
Motivated by answering an open problem in \cite[Remark 18]{CH2016} for stochastic Maxwell equations driven by additive noise,
 we establish the error analysis in mean-square sense of the stochastic Runge-Kutta method \eqref{RK method} in this part.

Recall that  the strong solution  of the stochastic Maxwell equations  \eqref{sM_equations} is
\begin{equation}\label{strong solution}
u(t_{n+1})=u(t_{n})+\int_{t_{n}}^{t_{n+1}}Mu(t){\rm d}t
+\int_{t_{n}}^{t_{n+1}}F(t,u(t)){\rm d}t
+\int_{t_{n}}^{t_{n+1}}B(t){\rm d}W(t).
\end{equation}
And
substituting  equation \eqref{Un} into \eqref{RK_compact_2} leads to the following formula of discrete solution
\begin{equation}\label{RK}
\begin{split}
u^{n+1}=&u^{n}+\tau\big(b^{T}\otimes M\big)\Big(I_{ 6s\times 6s}-\tau\big(A\otimes M\big)\Big)^{-1}\big({\bf 1}_{s}\otimes u^n\big)\\
&+\tau\big(b^{T}\otimes I\big)F^{n}(U_n)+\tau^2\big(b^{T}\otimes M\big)\Big(I_{ 6s\times 6s}-\tau\big(A\otimes M\big)\Big)^{-1}\big(A\otimes I\big)F^{n}(U_n)\\
&+\big(\widetilde{b}^{T}\otimes I\big)B^{n}\Delta W^{n+1}+\tau\big(b^{T}\otimes M\big)\Big(I_{ 6s\times 6s}-\tau\big(A\otimes M\big)\Big)^{-1}\Big(\big(\widetilde{A}\otimes I\big)B^{n}\Delta W^{n+1}\Big).
\end{split}
\end{equation}
Let $e^{n}=u(t_n)-u^n$. Substracting \eqref{RK} from \eqref{strong solution}, we obtain
\begin{align}
e^{n+1}=&e^{n}+\underbrace{\int_{t_{n}}^{t_{n+1}}Mu(t){\rm d}t-\tau\big(b^{T}\otimes M\big)\Big(I_{ 6s\times 6s}-\tau\big(A\otimes M\big)\Big)^{-1}\big({\bf 1}_{s}\otimes u^n\big)}_{I}\nonumber\\
&+\underbrace{\int_{t_{n}}^{t_{n+1}}F(t,u(t)){\rm d}t-\tau\big(b^{T}\otimes I\big)F^{n}(U_n)}_{II_{a}}\nonumber\\
&-\underbrace{\tau^2\big(b^{T}\otimes M\big)\Big(I_{ 6s\times 6s}-\tau\big(A\otimes M\big)\Big)^{-1}\big(A\otimes I\big)F^{n}(U_n)}_{II_{b}}\label{error}\\
&+\underbrace{\int_{t_{n}}^{t_{n+1}}B(t){\rm d}W(t)-\big(\widetilde{b}^{T}\otimes I\big)B^{n}\Delta W^{n+1}}_{III_{a}}\nonumber\\
&-\underbrace{\tau\big(b^{T}\otimes M\big)\Big(I_{ 6s\times 6s}-\tau\big(A\otimes M\big)\Big)^{-1}\Big(\big(\widetilde{A}\otimes I\big)B^{n}\Delta W^{n+1}\Big)}_{III_{b}}\nonumber\\
=:&e^{n}+I+II_{a}-II_{b}+III_{a}-III_{b}.\nonumber
\end{align}
Taking $\|\cdot\|^{2}_{\mathbb H}$-norm yields
\begin{equation}
\begin{split}
\|e^{n+1}\|^{2}_{\mathbb H}=&\|e^{n}\|^{2}_{\mathbb H}+\|I\|^{2}_{\mathbb H}+\|II\|^{2}_{\mathbb H}+\|III\|^2_{\mathbb H}+2\langle e^{n},I\rangle_{\mathbb H}+2\langle e^{n},II\rangle_{\mathbb H}+2\langle e^{n},III\rangle_{\mathbb H}\\
&+2\langle I,II\rangle_{\mathbb H}+2\langle I,III\rangle_{\mathbb H}
+2\langle II,III\rangle_{\mathbb H}\\
\leq&(1+\tau) \|e^{n}\|^{2}_{\mathbb H}+3\|I\|^{2}_{\mathbb H}+2\langle e^{n},I\rangle_{\mathbb H}+\Big(3+\frac{C}{\tau}\Big)\|II\|^{2}_{\mathbb H}+3\|III\|^2_{\mathbb H}+2\langle e^{n},III\rangle_{\mathbb H}.
\end{split}
\end{equation}

{\em Step 1.  The estimates of terms $\|I\|^{2}_{\mathbb H}$ and $\langle e^{n},I\rangle_{\mathbb H}$.} From \eqref{error}, we have
\begin{equation}
\begin{split}
I=&\underbrace{\int_{t_n}^{t_{n+1}}\big(Mu(t)-Mu(t_n)\big){\rm d}t}_{I_a}
+\tau Me^{n}\\
&+\underbrace{\tau Mu^{n}-\tau\big(b^{T}\otimes M\big)\Big(I_{ 6s\times 6s}-\tau\big(A\otimes M\big)\Big)^{-1}\big({\bf 1}_{s}\otimes u^n\big)}_{I_b}.
\end{split}
\end{equation}
From Proposition \ref{holder}, we know that
\[
{\mathbb E}\|I_a\|_{\mathbb H}^2\leq \tau\int_{t_n}^{t_{n+1}}{\mathbb E}\|u(t)-u(t_{n})\|^2_{{\mathcal D}(M)}{\rm d}t
\leq C\tau^3,
\]
and
\[
{\mathbb E}\|{\mathbb E}(I_a|{\mathcal F}_{t_n})\|_{\mathbb H}^2\leq \tau\int_{t_n}^{t_{n+1}}\|{\mathbb E}\big(u(t)-u(t_{n})|{\mathcal F}_{t_n}\big)\|^2_{{\mathcal D}(M)}{\rm d}t
\leq C\tau^4,
\]
where the constant $C$ depends on $T$, $\|B(t)\|_{HS(U,{\mathcal D}(M^2))}$ and $\|u_0\|_{L^2(\Omega,{\mathcal D}(M^2))}$.

From Proposition \ref{regularity} and the property of operator $M$, we know that
\[
\|\tau Me^{n}\|^2_{\mathbb H}=-\tau^2\langle e^{n}, M^2 e^{n} \rangle_{{\mathbb H}}
\leq \tau\|e^n\|^2_{\mathbb H}+C\tau^3 \Big(\|M^2 u(t_n)\|^2_{\mathbb H}+\|M^2 u^n\|^2_{\mathbb H}\Big)\leq \tau\|e^n\|^2_{\mathbb H}+C\tau^3 ,
\]
and
\[
\langle e^{n},\tau Me^{n} \rangle_{{\mathbb H}}=0,
\]
where the constant $C$ depends on $T$ and $\|Q^{\frac12}\|_{HS(U,H^2(D))}$.

Under the assumption $\sum_{i=1}^{s}b_i=1$, we know that
\[
\big(b^{T}\otimes I\big)\big({\bf 1}_{s}\otimes Mu^{n}\big)=(b^{T}{\bf 1}_{s})\otimes(IMu^n)=\Big(\sum_{i=1}^{s}b_i\Big)\otimes (Mu^n)=Mu^{n}.
\]
Since $b^{T}\otimes M=\big(b^{T}\otimes I\big)\big(I_{s\times s}\otimes M\big)$
and
$\big(I_{s\times s}\otimes M\big)\big(A\otimes M\big)=A\otimes M^2=\big(A\otimes M\big)\big(I_{s\times s}\otimes M\big)$, we have
\begin{equation}
\begin{split}
&\big(b^{T}\otimes M\big)\Big(I_{ 6s\times 6s}-\tau\big(A\otimes M\big)\Big)^{-1}\big({\bf 1}_{s}\otimes u^n\big)\\
&=\big(b^{T}\otimes I\big)\Big(I_{ 6s\times 6s}-\tau\big(A\otimes M\big)\Big)^{-1}\big(I_{s\times s}\otimes M\big)\big({\bf 1}_{s}\otimes u^n\big)\\
&=\big(b^{T}\otimes I\big)\Big(I_{ 6s\times 6s}-\tau\big(A\otimes M\big)\Big)^{-1}\big({\bf 1}_{s}\otimes Mu^{n}\big).
\end{split}
\end{equation}
Hence for term $I_b$, we get
\begin{equation}
\begin{split}
I_b=&\tau \big(b^{T}\otimes I\big)\big({\bf 1}_{s}\otimes Mu^{n}\big)
-\tau \big(b^{T}\otimes I\big)\Big(I_{ 6s\times 6s}-\tau\big(A\otimes M\big)\Big)^{-1}\big({\bf 1}_{s}\otimes Mu^{n}\big)\\
=&\tau\big(b^{T}\otimes I\big)\bigg[I_{6s\times 6s}-\Big(I_{ 6s\times 6s}-\tau\big(A\otimes M\big)\Big)^{-1}\bigg]\big({\bf 1}_{s}\otimes Mu^{n}\big).
\end{split}
\end{equation}
By Lemma \ref{est operator},
we get
\begin{equation*}
\begin{split}
\|I_b\|_{\mathbb H}\leq& C\tau \left\|\bigg[I_{6s\times 6s}-\Big(I_{ 6s\times 6s}-\tau\big(A\otimes M\big)\Big)^{-1}\bigg]\big({\bf 1}_{s}\otimes Mu^{n}\big)\right\|_{{\mathbb H}^s}\\
\leq& C\tau^2\|(A\otimes M)\big({\bf 1}_{s}\otimes Mu^{n}\big)\|_{{\mathbb H}^s}\\
=& C\tau^2\|(A{\bf 1}_s)\otimes M^2 u^{n}\|_{{\mathbb H}^s}
\leq C\tau^2\|u^n\|_{{\mathcal D}(M^2)},
\end{split}
\end{equation*}
and then
\[
{\mathbb E}\|I_b\|^2_{\mathbb H}\leq C\tau^4{\mathbb E}\|u^n\|_{{\mathcal D}(M^2)}^2\leq C\tau^4.
\]
Therefore,
\[
{\mathbb E}\|I\|^2_{\mathbb H}\leq\tau{\mathbb E}\|e^n\|^2_{\mathbb H}+ C\tau^3,\quad
{\mathbb E}\langle e^n,I\rangle_{{\mathbb H}}={\mathbb E}\langle e^n,{\mathbb E}\big(I_a|{\mathcal F}_{t_n}\big)\rangle_{{\mathbb H}}+{\mathbb E}\langle e^n,I_b\rangle_{{\mathbb H}}\leq \tau{\mathbb E}\|e^n\|^2_{\mathbb H}+C\tau^3.
\]

{\em Step 2. The estimate of the term $\|II\|_{\mathbb H}$ and $\langle e^n,II\rangle_{{\mathbb H}}$.}
For term $II_a$, we recall that $\sum_{i=1}^{s}b_{i}=1$,
\begin{equation}
\begin{split}
II_a=&\int_{t_n}^{t_{n+1}}\Big(F(t,u(t))-\sum_{i=1}^{s}b_iF(t_n+c_i\tau,U_{ni})\Big){\rm d}t
=\tau\Big(F(t_n,u(t_n)-F(t_n,u^n)\Big)\\
&+\int_{t_n}^{t_{n+1}}\Big(F(t,u(t))-F(t_n,u(t_n)\Big){\rm d}t
+\tau\sum_{i=1}^{s}b_i\Big(F(t_n,u^n)-F(t_n+c_i\tau,U_{ni})\Big).
\end{split}
\end{equation}
From the globally Lipschitz property of $F$, we have
\begin{equation}
\begin{split}
\|II_a\|^2_{\mathbb H}\leq C\tau^2\|e^n\|^2_{\mathbb H}+C\tau^4+C\tau\int_{t_n}^{t_{n+1}}
\|u(t)-u(t_n)\|^2_{\mathbb H}{\rm d}t
+C\tau^2\|U_{ni}-u^n\|^2_{\mathbb H}.
\end{split}
\end{equation}
The assertion (i) of Proposition \ref{holder} and the estimate for $U_{ni}-u^n$ in Proposition \ref{holder_RK} lead  to
\[
{\mathbb E}\|II_a\|^2_{\mathbb H}\leq C\tau^2{\mathbb E}\|e^n\|^2_{\mathbb H}+C\tau^3.
\]
The estimate of ${\mathbb E}\|{\mathbb E}(II_a|{\mathcal F}_{t_n})\|^2_{\mathbb H}$ is technical. In fact, take the term $$\int_{t_n}^{t_{n+1}}\Big(F(u(t))-F(u(t_n)\Big){\rm d}t$$ in $II_a$ as an example, where we let $F$ do not depend on time $t$ explicitly for ease of presentation, since the dependence on time causes no substantial problems in the analysis but just leads to longer formulas.

Thanks to Taylor formula, we have
\begin{equation}
\begin{split}
\int_{t_n}^{t_{n+1}}\Big(F(u(t))-F(u(t_n)\Big){\rm d}t
=&\int_{t_n}^{t_{n+1}}F^{\prime}(u(t_n))\big(u(t)-u(t_n)\big){\rm d}t\\
&+\frac12\int_{t_n}^{t_{n+1}} F^{\prime\prime}(u_{\theta})\Big(u(t)-u(t_{n}),~u(t)-u(t_n)\Big)  {\rm d}t,
\end{split}
\end{equation}
where $u_{\theta}$ is some point between $u(t_n)$ and $u(t)$.
The estimate of the second term on the above equation is based on the assertion (i) of Proposition \ref{holder}, which gives order $O(\tau^4)$ in mean-square sense. For the first term, we apply conditional expectation first,
\begin{equation}\label{4.43}
\begin{split}
{\mathbb E}\left( \int_{t_n}^{t_{n+1}}F^{\prime}(u(t_n))\big(u(t)-u(t_n)\big){\rm d}t\bigg|{\mathcal F}_{t_n}\right)
=\int_{t_n}^{t_{n+1}}F^{\prime}(u(t_n)){\mathbb E}\Big(\big(u(t)-u(t_n)\big)\Big|{\mathcal F}_{t_n}\Big){\rm d}t,
\end{split}
\end{equation}
where the adaptedness of $\{u(t)\}_{t\in[0,T]}$ and the properties of conditional expectation are used. Then by  the assertion (ii) of Proposition \ref{holder}, we know that \eqref{4.43} gives order $O(\tau^4)$ in mean-square sense.

Hence, by this approach we can show that
\[{\mathbb E}\|{\mathbb E}(II_a|{\mathcal F}_{t_n})\|^2_{\mathbb H}\leq C\tau^2{\mathbb E}\|e^n\|^2_{\mathbb H} +C\tau^4.\]

For term $II_b$, we have
\begin{equation}
\begin{split}
II_b=&\tau^2 \big(b^{T}\otimes I\big)\big(I_{s\times s}\otimes M\big)\Big(I_{6s\times 6s}-\tau\big(A\otimes M\big)\Big)^{-1}\big(A\otimes I\big)F^{n}(U_n)\\
=&\tau^2\big(b^{T}\otimes I\big)\Big(I_{6s\times 6s}-\tau\big(A\otimes M\big)\Big)^{-1}\big(I_{s\times s}\otimes M\big)\big(A\otimes I\big)F^{n}(U_n)\\
=&\tau^2\big(b^{T}\otimes I\big)\Big(I_{6s\times 6s}-\tau\big(A\otimes M\big)\Big)^{-1}\big(A\otimes I\big)\big(I_{s\times s}\otimes M\big)F^{n}(U_n),
\end{split}
\end{equation}
hence from \eqref{5.18}
\begin{equation}
\begin{split}
\|II_b\|_{\mathbb H}\leq& C\tau^2\|\Big(I_{6s\times 6s}-\tau\big(A\otimes M\big)\Big)^{-1}\big(A\otimes I\big)\big(I_{s\times s}\otimes M\big)F^{n}\|_{{\mathbb H}^s}\\
\leq& C\tau^2\|\big(A\otimes I\big)\big(I_{s\times s}\otimes M\big)F^{n}\|_{{\mathbb H}^s}
\leq C\tau^2\max_{1\leq i\leq s}\|F(t_n+C_i\tau, U_{ni})\|_{{\mathcal D}(M)}\\
\leq& C\tau^2\big(1+\|U_n\|_{{\mathcal D}(M)^s}\big),
\end{split}
\end{equation}
which leads to ${\mathbb E}\|II_b\|^2_{\mathbb H}\leq C\tau^4$.

Therefore,
\[
{\mathbb E}\|II\|^2_{\mathbb H}\leq C\tau^2{\mathbb E}\|e^n\|^2_{\mathbb H}+C\tau^3,
\]
and
\[
{\mathbb E}\langle e^n,II\rangle_{{\mathbb H}}={\mathbb E}\langle e^n,{\mathbb E}\big(II_a|{\mathcal F}_{t_n}\big)\rangle_{{\mathbb H}}-{\mathbb E}\langle e^n,II_b\rangle_{{\mathbb H}}\leq C\tau{\mathbb E}\|e^n\|^2_{\mathbb H}+C\tau^3.
\]

{\em Step 3. The estimate of the term $\|III\|_{\mathbb H}$.}
For term $III_a$, we recall that $\sum_{i=1}^{s}\widetilde{b}_i=1$,
\begin{equation}
\begin{split}
III_a=\int_{t_n}^{t_{n+1}}\Big(B(t)-\sum_{i=1}^{s}\widetilde{b}_iB^{ni}\Big){\rm d}W(t)
=\int_{t_n}^{t_{n+1}}\sum_{i=1}^{s}\widetilde{b}_i\Big(B(t)-B^{ni}\Big){\rm d}W(t),
\end{split}
\end{equation}
hence
\[
{\mathbb E}\|III_a\|^2_{\mathbb H}=\int_{t_n}^{t_{n+1}}\left\|\sum_{i=1}^{s}\widetilde{b}_i\Big(B(t)-B^{ni}\Big)\right\|^2_{HS(U_0,{\mathbb H})}{\rm d}t\leq C\tau^3.
\]
For term $III_b$, similarly to $II_b$, we have
\begin{equation*}
\begin{split}
III_b=&\tau\big(b^{T}\otimes I\big)\big(I_{s\times s}\otimes M\big)\Big(I_{6s\times 6s}-\tau\big(A\otimes M\big)\Big)^{-1}\Big(\big(\widetilde{A}\otimes I\big)B^{n}\Delta W^{n+1}\Big)\\
=&\tau\big(b^{T}\otimes I\big)\Big(I_{6s\times 6s}-\tau\big(A\otimes M\big)\Big)^{-1}\big(I_{s\times s}\otimes M\big)\Big(\big(\widetilde{A}\otimes I\big)B^{n}\Delta W^{n+1}\Big)\\
=&\tau\big(b^{T}\otimes I\big)\Big(I_{6s\times 6s}-\tau\big(A\otimes M\big)\Big)^{-1}\big(\widetilde{A}\otimes I\big)\big(I_{s\times s}\otimes M\big)\big(B^{n}\Delta W^{n+1}\big),
\end{split}
\end{equation*}
hence from \eqref{5.18}
\begin{equation}
\begin{split}
{\mathbb E}\|III_b\|^2_{\mathbb H}\leq& C\tau^2\left\|\Big(I_{6s\times 6s}-\tau\big(A\otimes M\big)\Big)^{-1}\big(\widetilde{A}\otimes I\big)\big(I_{s\times s}\otimes M\big)\big(B^{n}\Delta W^{n+1}\big)\right\|^2_{{\mathbb H}^s}\\
\leq & C\tau^2\left\|\big(\widetilde{A}\otimes I\big)\big(I_{s\times s}\otimes M\big)\big(B^{n}\Delta W^{n+1}\big)\right\|^2_{{\mathbb H}^s}\\
\leq & C\tau^3.
\end{split}
\end{equation}
Therefore,
\[
{\mathbb E}\|III\|^2_{\mathbb H}\leq C\tau^3,\quad
{\mathbb E}\langle e^n,III\rangle_{{\mathbb H}}=0.
\]

{\em Step 4. Application of Gronwall's inequality.}
Combining all the estimates in Steps 1-3, we get
\[
{\mathbb E}\|e^{n+1}\|^2_{\mathbb H}\leq (1+C\tau){\mathbb E}\|e^{n}\|^2_{\mathbb H}+C\tau^3,
\]
which by Growall's inequality leads to
\[
\sup_{0\leq n\leq N}\Big({\mathbb E}\|e^{n}\|^2_{\mathbb H}\Big)^{\frac12}\leq C\tau.
\]
The above result is stated in the following theorem.
\begin{theorem}\label{estimate_e_k_RK}
In addition to the  conditions of Proposition \ref{prop} with $k=2$, let $\sum_{i=1}^{s}b_i=\sum_{i=1}^{s}\widetilde{b}_i\equiv 1$.  we have for the discrete solution of stochastic Runge-Kutta method \eqref{RK method_1}-\eqref{RK method_2},
	\begin{equation}
	\begin{split}
	\max_{1\leq n\leq N}\big(\mathbb{E}\|u(t_n)-u^n\|_{\mathbb H}^2\big)^{\frac12}\leq C\tau,
	\end{split}
	\end{equation}
	where the positive constant $C$ depends on the Lipschitz coefficients of $F$ and $B$, $T$, $\|u_0\|_{L^2(\Omega;{\mathcal D}(M^2))}$ and $\sup_{t\in[0,T]}\|B(t)\|_{HS(U,{\mathcal D}(M^2))}$, but independent of $\tau$ and $n$.
\end{theorem}
We observe that the Butcher Tableaux of the implicit Euler method and the midpoint method satisfy algebraic stability and the coercivity condition, therefore the mean-square convergence order of these two examples is of one,
\begin{corollary}
	Under the same assumptions of Theorem \ref{estimate_e_k_RK}. For implicit Euler method, or for midpoint method we have
	\begin{equation}
	\begin{split}
	\max_{1\leq k\leq N}\big(\mathbb{E}\|u(t_k)-u^k\|_{\mathbb H}^2\big)^{\frac12}\leq C\tau,
	\end{split}
	\end{equation}
	where the positive constant $C$ depends on the Lipschitz coefficients of $F$ and $B$, $T$, $\|u_0\|_{L^2(\Omega;{\mathcal D}(M^2))}$ and $\sup_{t\in[0,T]}\|B(t)\|_{HS(U,{\mathcal D}(M^2))}$, but independent of $\tau$ and $k$.
\end{corollary}

\bibliographystyle{plain}
\bibliography{maxwell}






\end{document}